\documentclass[reqno]{amsart}

\input xy
\xyoption{all}
\usepackage{accents}
\usepackage{epsfig}
\usepackage{color}
\usepackage{amsthm}
\usepackage{amssymb}
\usepackage{amsmath}
\usepackage{amscd}
\usepackage{amsopn}
\usepackage{graphicx}
\usepackage{centernot}

\usepackage{xspace}

\usepackage{url}
\usepackage{enumitem, hyperref}\hypersetup{colorlinks}

\usepackage{color} 

\definecolor{darkred}{rgb}{1,0,0} 
\definecolor{darkgreen}{rgb}{0,0.6,0}
\definecolor{darkblue}{rgb}{0,0,.8}

\hypersetup{colorlinks,
linkcolor=darkblue,
filecolor=darkgreen,
urlcolor=darkred,
citecolor=darkgreen}

\makeatletter
\def\reflb#1#2{\begingroup
    #2%
    \def\@currentlabel{#2}%
    \phantomsection\label{#1}\endgroup
}
\makeatother

\numberwithin{equation}{section}
\newtheorem {Theorem}{Theorem}
\numberwithin{Theorem}{section}

\newtheorem {Lemma}[Theorem]    {Lemma}

\newtheorem {Proposition}[Theorem]{Proposition}
\newtheorem {Corollary}[Theorem]{Corollary}
\theoremstyle{definition}

\theoremstyle{remark}
\newtheorem{Remark}[Theorem]{Remark}
\newtheorem{Example}[Theorem]{Example}

\def    \eps    {\epsilon}

\newcommand{\CA}{{\mathcal A}}

\newcommand{\CS}{{\mathcal S}}

\newcommand{\supp}{\operatorname{supp}}

\newcommand{\sgn}{\operatorname{sgn}}

\newcommand{\id}{{\mathit id}}

\newcommand{\st}{{\mathit st}}
\newcommand{\const}{{\mathit const}}

\newcommand{\ff}{{\mathfrak f}}
\newcommand{\fh}{{\mathfrak h}}

\newcommand{\tx}{\tilde{x}}

\newcommand{\hmu}{\hat{\mu}}

\newcommand{\tH}{\tilde{H}}

\newcommand{\Mm}{{\mathcal M}}

\newcommand{\PP}{{\mathcal P}}

\def    \C      {{\mathbb C}}
\def    \R      {{\mathbb R}}

\def    \Z      {{\mathbb Z}}
\def    \N      {{\mathbb N}}
\def    \Q      {{\mathbb Q}}

\def    \CP     {{\mathbb C}{\mathbb P}}

\def    \12    {{\frac{1}{2}}}

\def    \p      {\partial}

\def    \im     {\operatorname{im}}

\def    \SH     {\operatorname{SH}}

\def    \HF     {\operatorname{HF}}
\def    \HC     {\operatorname{HC}}

\def    \H     {\operatorname{H}}

\def    \CF     {\operatorname{CF}}

\def    \bx     {\bar{x}}

\def    \MUM  {\operatorname{\mu_{\scriptscriptstyle{M}}}}

\def    \s  {\operatorname{c}}

\def\wh{\widehat}

\newcommand{\tpi}{\tilde{\pi}}

\usepackage{mathtools}
\def   \Dlim {\operatorname{\varinjlim}}

\usepackage{tikz}
\usetikzlibrary{shapes,shadows,arrows}
\usepackage{tikz-cd}

\usepackage{xspace}

\begin{document}


\setlength{\smallskipamount}{6pt}
\setlength{\medskipamount}{10pt}
\setlength{\bigskipamount}{16pt}





\title[Symplectic Homology of Prequantization Bundles]{On the Filtered
  Symplectic Homology of Prequantization Bundles}

\author[Viktor Ginzburg]{Viktor L. Ginzburg}
\author[Jeongmin Shon]{Jeongmin Shon}

\address{GS and VG: Department of Mathematics, UC Santa Cruz, Santa Cruz, CA
  95064, USA} \email{ginzburg@ucsc.edu}\email{jeshon@ucsc.edu}



\date{\today} 

\bigskip

\begin{abstract}
  We study Reeb dynamics on prequantization circle bundles and the
  filtered (equivariant) symplectic homology of prequantization line
  bundles, aka negative line bundles, with symplectically aspherical
  base. We define (equivariant) symplectic capacities, obtain an upper
  bound on their growth, prove uniform instability of the filtered
  symplectic homology and touch upon the question of stable
  displacement. We also introduce a new algebraic structure on the
  positive (equivariant) symplectic homology capturing the free
  homotopy class of a closed Reeb orbit -- the linking number
  filtration -- and use it to give a new proof of the non-degenerate
  case of the contact Conley conjecture (i.e., the existence of
  infinitely many simple closed Reeb orbits), not relying on contact
  homology.
\end{abstract}

\maketitle

\tableofcontents

\section{Introduction}
\label{sec:intro+results}
In this paper we study dynamics of Reeb flows on prequantization
$S^1$-bundles and the filtered (equivariant) symplectic homology of
the associated prequantization line bundles $E$, aka negative line
bundles, over symplectically aspherical manifolds. The new features in
this case, as compared to the symplectic homology of exact fillings,
come from the difference between the Hamiltonian action (i.e., the
symplectic area), giving rise to the action filtration of the
homology, and the contact action.  We define equivariant symplectic
capacities, obtain an upper bound on their growth, prove uniform
instability of the filtered symplectic homology, and touch upon the
question of stable displacement in $E$. We also introduce a new
algebraic structure on the positive (equivariant) symplectic homology
capturing the free homotopy class of a closed Reeb orbit -- the
linking number filtration. We then use this filtration to give a new
proof of the non-degenerate case of the contact Conley conjecture, not
relying on contact homology.

Prequantization $S^1$-bundles $M$ form an interesting class of
examples to study Reeb dynamics, and, in particular, the question of
multiplicity of closed Reeb orbits with applications, for instance, to
closed magnetic geodesics and geodesics on CROSS's; see, e.g.,
\cite{GGM:CC2}. The range of possible dynamics behavior in this case
is similar to that for Hamiltonian diffeomorphisms and more limited
than for all contact structures where it should be, more adequately,
compared with the class of symplectomorphisms; see
\cite{GG:CC}. Furthermore, in many instances, various flavors of the
Floer-type homology groups associated to $M$ can be calculated
explicitly providing convenient basic tools to study the multiplicity
questions.

Most of the Floer theoretic constructions counting closed Reeb orbits
require a strong symplectic filling $W$ of $M$ and, in general, the
properties of the resulting groups depend on the choice of $W$ unless
$W$ is exact and $c_1(TW)=0$. (The exceptions are the cylindrical
contact homology and the contact homology linearized by an
augmentation, but here we are only concerned with the symplectic
homology.) The most natural filling $W$ of a prequantization 
$S^1$-bundle $M$ is that by the disk bundle or, to be more precise, by the
region bounded by $(M,\alpha)$, where $\alpha$ is the contact form, in
the line bundle $E$. 

However, the filling $W$ is also quite awkward to work with. The main
reason is that $W$ is never exact, although it is aspherical when the
base $B$ is aspherical. As a consequence, the Hamiltonian and contact
actions of closed Reeb orbits differ and the Hamiltonian action,
giving rise to the action filtration on the homology, is not
necessarily non-negative; see Proposition \ref{prop:negative}. The
second difficulty comes from that the natural map
$\pi_1(M)\to \pi_1(W)$ fails, in general, to be one-to-one. This is
the case, for instance, when $B$ is symplectically aspherical: the
fiber, which is not contractible in $M$, becomes contractible in
$W$. This fact has important conceptual consequences. For instance,
the proof of the contact Conley conjecture for prequantization bundles
with aspherical base (i.e., the existence of infinitely many simple
closed Reeb orbits) from \cite{GGM:CC,GGM:CC2}, which relies on the free
homotopy class grading of the cylindrical contact homology, fails to
directly translate to the symplectic homology framework.

One of the goals of this paper is to systematically study the filtered
symplectic homology of $W$, equivariant and ordinary. Note that total
Floer theoretic ``invariants'' such as the equivariant and/or positive
symplectic homology of $E$ has been calculated explicitly; see
\cite{GGM:CC2,Oa:Leray-Serre,Ri:AM}. Moreover, in many cases the
total symplectic homology of $E$ vanishes and this fact alone is
sufficient for many applications. The results of the paper comprise a
part of the second author's Ph.D. thesis, \cite{Sh}.

The paper is organized as follows. In Section \ref{sec:prelim}, we set
our conventions and notation and very briefly recall the constructions
of various flavors of symplectic homology. The only non-standard point
here is the definition of the negative/positive symplectic
homology. Namely, since the filling is not required to be exact, this
homology cannot be defined as the subcomplex generated by the orbits
with negative action. Instead, following \cite{BO17}, it is defined
essentially as the homology of the subcomplex generated by the
constant one-periodic orbits of an admissible Hamiltonian. In Section
\ref{sec:vanishing}, we investigate the consequences of vanishing of
the total symplectic homology. We introduce a class of (equivariant)
symplectic capacities, prove upper bounds on their growth, show that
vanishing is equivalent to a seemingly stronger condition of uniform
instability, and revisit the relation between vanishing of the
symplectic homology and displacement. In Section \ref{sec:prequant},
we specialize these results to prequantization line bundles with
symplectically aspherical base and also briefly touch upon the
question of stable displacement in prequantization bundles.  A new
algebraic structure on the positive (equivariant) symplectic homology
of such bundles -- the linking number filtration -- is introduced in
Section \ref{sec:linking}, where we also calculate the associated
graded homology groups. This filtration is given by the linking number
of a closed Reeb orbit with the base $B$. It is then used in Section
\ref{sec:CCC} to reprove the non-degenerate case of the contact Conley
conjecture circumventing the foundational difficulties inherent in the
construction of the contact homology.

\medskip
\noindent\textbf{Acknowledgments}
The authors are grateful to Frederic Bourgeois, River Chiang, Yasha
Eliashberg, Ba\c sak G\"urel, Kei Irie, Leonardo Macarini and Alex
Oancea for useful discussions. Parts of this work were carried out
during the first author's visit to the National Center for Theoretical
Sciences (NCTS, Taipei, Taiwan) and he would like to thank the Center
for its warm hospitality and support.

\section{Preliminaries}
\label{sec:prelim}
In this section we set our conventions and notation used throughout
the paper and briefly discuss various flavors of contact homology.
\subsection{Conventions}
\label{sec:conventions}

The conventions and notation adopted in this paper are similar to
those used in \cite{GG:convex}, resulting ultimately, through
cancellations of several negative signs, in the same grading and the
action filtration as in \cite{BO:Gysin,BO17}.

Throughout the paper, we usually assume, unless specified otherwise
that the underlying symplectic manifold $(W^{2n},\omega)$ is
symplectically aspherical, i.e.,
$[\omega]|_{\pi_2(W)}=0=c_1(TW)|_{\pi_2(W)}$ and compact with contact
type boundary $M=\p W$. The \emph{symplectic completion} of $W$ is
$$
\widehat{W}=W\cup_{M} \big(M\times [1,\infty)\big),
$$ 
equipped with the symplectic form $\omega$ extended to the cylindrical
part as $d(r\alpha)$, where $\alpha$ is a contact primitive of
$\omega$ on $M$ and $r$ is the coordinate on $[1,\infty)$.

We denote by $\PP(\alpha)$ or $\PP(W)$ the set of contractible in $W$
periodic orbits $x$ of the Reeb flow on $(M^{2n-1},\alpha)$. Since the
form $\omega$ is not required to be exact there are two different ways
to define an action of $x$. One is the \emph{contact action}
$\CA_\alpha(x)$ given by the integral of $\alpha$ over $x$, i.e., the
period of $x$ as a closed orbit of the Reeb flow. The second one is
the \emph{symplectic or Hamiltonian action} $\CA_\omega(x)$. This is
the symplectic area bounded by $x$, i.e., the integral of $\omega$
over a disk bounded by $x$ in $W$. By Stokes' formula,
$\CA_\alpha(x)=\CA_\omega(x)$ when $x$ is contractible in $M$, but in
general $\CA_\alpha(x)\neq\CA_\omega(x)$. We denote the resulting
contact and symplectic action spectra by $\CS(\alpha)$ and
$\CS_\omega(W)$, respectively.

The circle $S^1=\R/\Z$ plays several different roles throughout the
paper. We denote it by $G$ when we want to emphasize the role of the
group structure on~$S^1$. 

The Hamiltonians $H$ on $\wh{W}$ are always required to be
one-periodic in time, i.e., $H\colon S^1\times \wh{W}\to \R$. In fact,
most of the time the Hamiltonians we are actually interested in are
\emph{autonomous}, i.e., independent of time. The (time-dependent)
Hamiltonian vector field $X_H$ of $H$ is given by the Hamilton
equation $i_{X_H}\omega=-dH$. For instance, on the cylindrical part
$M\times [1,\infty)$ with $\omega=d(r\alpha)$ the Hamiltonian vector
field of $H=r$ is the Reeb vector field of $\alpha$.

We focus on contractible one-periodic orbits of $H$. Such orbits can
be identified with the critical points of the \emph{action functional}
$\CA_H\colon \Lambda\to\R$ on the space $\Lambda$ of contractible
loops $x$ in $\wh{W}$ given by
\begin{equation}
\label{eq:action-0}
\CA_{H}(x) =\CA_\omega(x)-\int_{S^1} H(t,x(t))\,dt.
\end{equation}
We will always require $H$ to have the form $H=\kappa r+c$, where
$\kappa\not\in\CS_\alpha(\alpha)$, outside a compact set. Under this
condition the Floer homology of $H$ is defined; see, e.g.,
\cite{Vi:GAFA}. Note, however, that the homology depends on $\kappa$.

The \emph{action spectrum} of $H$, i.e., the collection of action
values for all contractible one-periodic orbits of $H$, will be
denoted by $\CS(H)$.  When $H$ is autonomous, a one-periodic orbit $y$
is said to be a \emph{reparametrization} of $x$ if $y(t)=x(t+\theta)$
for some $\theta\in G=S^1$.  Two one-periodic orbits are said to be
\emph{geometrically distinct} if one of them is not a
reparametrization of the other. We denote by $\PP(H)$ the collection
of all geometrically distinct contractible one-periodic orbits of $H$.

With our sign conventions in the definitions of $X_H$ and $\CA_H$, the
Hamiltonian actions on $x$ converge to $\CS_\omega(x)$ in the
construction of the symplectic homology. In particular,
$\CS(H)\to\CS_\omega(W)\cup\{0\}$.

We normalize the \emph{Conley--Zehnder index}, denoted throughout the
paper by $\mu$, by requiring the flow for $t\in [0,\,1]$ of a small
positive definite quadratic Hamiltonian $Q$ on $\R^{2n}$ to have index
$n$. More generally, when $Q$ is small and non-degenerate, the flow
has index equal to $(\sgn Q)/2$, where $\sgn Q$ is the signature of
$Q$.  In other words, the Conley--Zehnder index of a non-degenerate
critical point $x$ of a $C^2$-small autonomous Hamiltonian $H$ on
$W^{2n}$ is equal to $n-\MUM$, where $\MUM=\MUM(H)$ is the Morse index
of $H$ at $x$. The \emph{mean index} of a periodic orbit $x$ will be denoted
by $\hmu(x)$; see, e.g., \cite{Lo,SZ} for the definitions and also
\cite{GG:convex} for additional references and a detailed discussion.

We denote by $\HF(H)$ the Floer homology of $H$ (when it is defined)
and by $\HF^I(H)$ the filtered Floer homology, where $I\subset \R$ is
an interval, possibly infinite, with end points not in
$\CS(H)$. Likewise, the (filtered) $G=S^1$-equivariant Floer homology
is denoted by $\HF^G(H)$ and $\HF^{G,I}(H)$; see, e.g.,
\cite{BO:Gysin,GG:convex} and references therein.  Note that for
$I=\R$ the filtered homology groups turn into total Floer homology
$\HF(H)$ and $\HF^G(H)$. These homology groups are graded by the
Conley--Zehnder index, but the grading is suppressed in the notation
when it is not essential.

Our choice of signs in \eqref{eq:action-0} effects the signs in the
Floer equation. Recall that the Floer equation is the
$L^2$-anti-gradient flow equation for $\CA_H$ on $\Lambda$ with
respect to a metric $\left<\cdot\,,\cdot\right>$ on $\wh{W}$
compatible with $\omega$:
$$
\p_s u =-\nabla_{L^2}\CA_H(u),
$$
where $u\colon \R\times S^1\to V$ and $s$ is the coordinate on
$\R$. Explicitly, this equation has the form
\begin{equation}
\label{eq:Floer-0}
\p_s u + J\p_t u=\nabla H,
\end{equation}
where $t$ is the coordinate on $S^1$. Here the almost complex
structure $J$ is defined by the condition
$\left<\cdot\,,\cdot\right>=\omega(J\cdot\,,\cdot)$ making $J$ act on
the first argument in $\omega$ rather than the second one, which is
more common, to ensure that the left hand side of the Floer equation
is still the Cauchy--Riemann operator $\bar\p_J$. (Thus $J=-J_0$,
where $J_0$ is defined by acting on the second argument in $\omega$,
and $X_H=-J\nabla H$.) Note, however, that now the right hand side of
\eqref{eq:Floer-0} is $\nabla H$ with positive sign.

The differential in the Floer or Morse complex is defined by counting
the downward Floer or Morse trajectories. As a consequence, a monotone
increasing homotopy of Hamiltonians induces a continuation map
preserving the action filtration in homology. (Clearly, $H\geq K$ on
$S^1\times \wh{W}$ if and only if $\CA_H\leq \CA_K$ on $\Lambda$.)
Thus we have natural maps $\HF^I(K)\to \HF^I(H)$ and similar maps in
the equivariant setting.

\subsection{Different flavors of symplectic homology}
\label{sec:sympl_hom}
Our main goal in this section is to recall the definition of several
kinds of symplectic homology groups associated with a compact
symplectic manifold $W$ with contact type boundary $M$. Our treatment
of the subject is intentionally brief, for the most part the material
is standard or nearly standard, and we refer the reader to numerous
other sources for a more detailed discussion; see, e.g.,
\cite{BO:Duke,BO:Exact,BO:Gysin,BO17,CO,GG:convex,Se,Vi:GAFA} and
references therein. Throughout the paper, all homology groups are
taken with rational coefficients unless specifically stated
otherwise. This choice of the coefficient field is essential, although
suppressed in the notation, and some of the results are simply not
true when, say, the coefficient field has finite characteristic.

Recall that a Hamiltonian $H\colon S^1\times \wh{W}\to \R$ is said to
be \emph{admissible} if $H=\kappa r+c$, where
$\kappa\not\in\CS_\alpha(\alpha)$, outside a compact set and $H\leq 0$
on $W$. The former condition guarantees that the (equivariant) Floer
homology of $H$ is defined. The \emph{(equivariant) filtered
  symplectic homology} groups of $W$ are, by definition, the limits
$$
\SH^I(W)=\varinjlim_H\HF^I(H)
$$
and
$$
\SH^{G,I}(W)=\varinjlim_H\HF^{G,I}(H)
$$
taken over all admissible Hamiltonians $H$.  It is sufficient to take
the direct limit over a cofinal sequence.  For instance, we can
require that $H=\const<0$ on $W$ and that $H$ depends only on $r$ on
the cylindrical part. Just as the Floer homology, the symplectic
homology groups are graded by the Conley--Zehnder index. The
\emph{total} (equivariant) symplectic homology groups $\SH(W)$
(respectively, $\SH^G(W)$) are obtained by setting $I=\R$.

Note that the choice of the contact primitive $\alpha$ on $M$
implicitly enters the definition of the homology. However, one can
show that the filtered homology is independent of $\alpha$; see
\cite{Vi:GAFA}.

The (equivariant) homology comprises roughly speaking two parts: one
-- the ``negative'' homology -- coming from the constant orbits of
$H$ and the other -- the positive'' part -- generated by the
non-constant orbits. Let us describe the construction in the
non-equivariant setting. The equivariant case can be dealt with in a
similar fashion.

Assume first that $\alpha$ on $M$ is non-degenerate. This can always
be achieved by replacing $M$ by its $C^\infty$-small perturbation. Let
$H$ be a non-positive $C^2$-small Morse function on $W$ and a monotone
increasing, convex function $h$ of $r$ on the cylindrical part such
that $h''>0$ in the region containing one-periodic orbits of
$H$. Clearly, $H$ is a Morse--Bott non-degenerate Hamiltonian and the
Hamiltonians meeting the above conditions form a cofinal family. Let
$\CF^-(H)$ be the subspace in the Floer complex, or, to be more
precise, the Morse--Bott Floer complex $\CF(H)$ of $H$ generated by
the critical points of $H$; see, e.g., \cite{BO:Duke}. The key
observation now is that $\CF^-(H)$ is actually a subcomplex of
$\CF(H)$; see \cite[Rmk.\ 2]{BO17}. This is absolutely not obvious and
the reason is that, as shown in \cite[p.\ 654]{BO:Exact} (see also
\cite[Lemma 2.3]{CO}), a Floer trajectory asymptotic at $+\infty$ to a
periodic orbit on a level $r=r_0$ cannot stay entirely in the union of
$W$ with the domain $r\leq r_0$. Therefore, by the standard maximum
principle such a trajectory cannot be asymptotic to a critical point
of $H$ in $W$ at $-\infty$. Hence $\CF^-(H)$ is closed under the Floer
differential.  Now we can interpret
$$
\CF^+(H)=\CF(H)/\CF^-(H),
$$
as a Morse--Bott type Floer complex arising from the non-trivial
one-periodic orbits of $H$. We denote the resulting negative/positive
Floer homology by $\HF^\pm(H)$.

Passing to the limit over $H$, we obtain the \emph{negative/positive}
symplectic homology groups $\SH^\pm(W)$, which fit into the long exact
sequence
\begin{equation}
\label{eq:exact-seq}
\ldots \to \SH^-(W)\to \SH(W)\to \SH^+(W)\to \ldots .
\end{equation}

When $\alpha$ is degenerate, we approximate it by non-degenerate forms
$\alpha'$ (or, equivalently, approximate $W$ by small perturbations
$W'$ in $\wh{W}$ with non-degenerate, in the obvious sense,
characteristic foliation), and pass to the limit as
$\alpha'\to\alpha$. It is easy to see that the resulting
negative/positive homology is well defined and we still have the long
exact sequence \eqref{eq:exact-seq}. We will give a slightly different
description of the positive homology in Section \ref{sec:linking-def}.

Note that this construction works essentially under no restrictions on
$W$ as long as the Floer homology of $H$ is defined; e.g., $W$ can be
weakly monotone. Moreover, the Bourgeois-Oancea ``maximum principle''
argument also gives a filtration of the Floer homology of $H$ by the
``level of $r$'' with $\HF^-(H)$ lying the lowest
level. (Alternatively, one can use the approach from
\cite[Appendix D]{McLR}.)

The homology $\HF^\pm(H)$ inherits the action filtration in the
obvious way and thus we have the groups $\HF^{\pm,I}(H)$, where the
end points of $I$ are required to be outside $\CS_\omega(H)$. As a
consequence, we obtain the filtered groups $\SH^{\pm,I}(W)$ with the
end points of $I$ outside $\CS_\omega(W)\cup\{0\}$. Occasionally, we
will use the notation $\SH^{\pm,(-\infty,\,0]}(W)$ (respectively,
$\SH^{(-\infty,\,0]}(W)$) for the inverse limit of the groups
$\SH^{\pm,(-\infty,\,\eps]}(W)$ (respectively,
$\SH^{(-\infty,\,\eps]}(W)$) as $\eps\searrow 0$. 

The long exact sequence \eqref{eq:exact-seq} still holds for the
filtered groups $\SH^{\pm,I}(H)$ and $\SH^I(W)$.  Since the constant
orbits of $H$ have non-positive action, $\SH^{-,I}(W)=0$ if
$I\subset (0,\,\infty)$ and there is a natural map
$$
\SH^-(W)\to \SH^{(-\infty,\,0]}(W).
$$
When the form $\omega$ is exact, this map is an isomorphism. Although
we do not have an explicit example, there seems to be no reason to
expect this map to be either one-to-one or onto when $\omega$ is just
aspherical; cf.\ Proposition \ref{prop:negative} and
\cite{Oa:Leray-Serre}. (Hence the terms ``negative/positive''
symplectic homology is somewhat misleading.)

These constructions extend to the equivariant setting in the standard
way.

A word is also due on the functoriality and invariance of symplectic
homology. When $W$ is exact the subject is quite standard and treated
in detail and greater generality in numerous papers starting with
\cite{Vi:GAFA}; see, e.g., \cite{CO,Gutt,Se} and references therein.
However, without some form of the exactness condition, the questions
of functoriality and invariance of the homology are less
understood. In fact, already when $W$ is aspherical, functoriality
becomes a rather delicate question and we opt here for the minimalist
approach to it which however is sufficient for our purposes.

Let $X$ be a Liouville vector field on $\wh{W}\setminus Z$, where
$Z\subset W$ is a compact set, which agrees with $r\p_r$ on the
cylindrical part of $\wh{W}$. (For instance, we can take as $Z$ a
closed codimension-two submanifold such that $[Z]\in \H_{2n-2}(W)$ is
Poincar\'e dual to the relative cohomology class in $\H^2(W,M)$ of the
pair $(\omega,\alpha_0)$.) Let $M_\tau$, $\tau\in [0,\,1]$, be a
family of closed hypersurfaces smoothly depending on $\tau$, enclosing
$Z$ and transverse to $X$. We denote by $W_\tau\supset Z$ the domain
in $\wh{W}$ bounded by $M_\tau$. Assume furthermore that $M_0=M$, and
hence $W_0=W$. Then, as is easy to see, the total (equivariant,
positive, negative, etc.) symplectic homology of $W_\tau$ is
independent of $\tau$; see \cite{Vi:GAFA}. For the filtered homology,
this is no longer true. However, when $W_1\supset W$, there is a
natural map from the filtered homology of $W_1$ to the filtered
homology of $W$, a very particular case of the ``Viterbo transfer''
morphism, which is essentially given by a monotone homotopy of the
Hamiltonians. On the level of total homology, this map is an
isomorphism.

The relations between the equivariant and non-equivariant symplectic
homology are similar to the relations between the ordinary equivariant
and non-equivariant homology.

On the one hand, we have the Gysin exact sequence for the
positive/negative and filtered symplectic homology (see
\cite{BO:Exact,BO:Gysin}):
$$
\ldots\to \SH_*^{\star}(W)\to \SH_*^{\star, G}(W)\stackrel{D}{\to}
\SH_{*-2}^{\star, G}(W)\to \SH_{*-1}^{\star}(W)\to\ldots
$$
where $\star=\{\pm,I\}$ (in all combinations including $\star=I$) and
we refer to $D$ as the \emph{shift operator}; see also
\cite{GG:convex}. As a consequence, vanishing of $\SH^{\star,G}(W)$
implies vanishing of $\SH^\star(W)$.

On the other hand, there is a Leray--Serre type spectral sequence
starting with $E^2=\SH_*^\star(W)\otimes\H_*(\CP^\infty)$, where
$\star=\pm$ or nothing, and converging to $\SH^{\star,G}(W)$; see
\cite{Hu,BO:Gysin,BO17,Se,Vi:GAFA}. As a consequence, vanishing of
$\SH_*^\star(W)$ (with $\star=\pm$ or nothing) implies and is, by the
Gysin sequence, equivalent to vanishing of $\SH^{\star,G}(W)$.

\section{Vanishing of symplectic homology and its consequences}
\label{sec:vanishing}
In this section we analyze general quantitative and qualitative
consequences of vanishing of the symplectic homology, focusing mainly on
symplectically aspherical manifolds.

\subsection{Homology calculations and equivariant capacities}
\label{sec:hom-cap}
The condition that $\SH(W)=0$ readily lends itself for an explicit
calculation of the (equivariant) positive symplectic homology, which
then can be used to define several variants of the homological symplectic
capacities.  We start with a calculation of the negative and positive
equivariant symplectic homology of $W$.

\begin{Proposition}
\label{prop:hom-calc}
Assume that $W^{2n}$ is symplectically aspherical and $\SH(W)=0$. Then
we have the following natural isomorphisms:
\begin{itemize}
\item[\rm{(i)}] $\SH^-(W)=\H_*(W,\p W)[-n]$ and $\SH^{-,G}(W)=\H_*(W,\p
  W)\otimes \H_*(\CP^\infty)[-n]$;
\item[\rm{(ii)}] $\SH^+(W)=\H_*(W,\p W)[-n+1]$ and 
\begin{equation}
\label{eq:tensor}
\SH^{+,G}(W)=\H_*(W,\p W)\otimes \H_*(\CP^\infty)[-n+1];
\end{equation}
\item[\rm{(iii)}] combined with the identification \eqref{eq:tensor}, the
  Gysin sequence shift map
$$
\SH^{+,G}_{r+2}(W)\stackrel{D}{\longrightarrow}\SH^{+,G}_{r}(W)
$$
is the identity on the first factor and the map
$\H_{q+2}(\CP^\infty)\to \H_{q}(\CP^\infty)$ on the second, given by
the pairing with a suitably chosen generator of
$\H^{2}(\CP^\infty)$. In particular, $D$ is an isomorphism when
$r\geq n+1$.
\end{itemize}
\end{Proposition}

\begin{proof}
  Assertion (i) is an immediate consequence of the definitions and the
  condition that $W$ is symplectically aspherical; see, e.g.,
  \cite{BO:Gysin,Vi:GAFA}. With our grading conventions (which
  ultimately result in the same grading as in \cite{BO:Gysin}), we
  have
\begin{align*}
& \SH^-_r(W) = \H_{n+r}(W, \p W) ,\\
& \SH^{-,G}_r(W) = \bigoplus_{p+q=r} \H_{p+n}(W, \p W) \otimes \H_q(\C P^{\infty}),
\end{align*}
where all homology groups are taken with coefficients in $\Q$. In
particular, as $W$ is oriented,
\begin{align*}
& \SH^-_n(W) = \Q ,\\
& \SH^-_q(W) = 0  \,\,\, \mbox{if} \,\,\, q\geq n+1
\end{align*}
Combining the assumption $\SH(W) = 0$ with the long
exact sequence 
\begin{center}
\begin{tikzpicture}
\node (a1) {$\SH(W)$} ;
\node [right of = a1, xshift=30mm](a2) {$\SH^+(W)$} ;
\node [below of = a1, yshift=-12mm](b1) {};
\node [right of =b1, xshift=10mm](b2) {$\SH^-(W)$};
\path[draw, ->](a1) -- node[yshift=3mm]{} (a2);
\path[draw, ->](b2) -- node[yshift=0mm, xshift=-4mm]{} (a1);
\path[draw, ->](a2) -- node[yshift=0mm, xshift=5mm]{$[-1]$} (b2);
\end{tikzpicture}
\end{center}
we see that $\SH_{q+1}^+(W) = \SH_{q}^-(W)$. Hence, we have
\begin{align}\label{zero}
& \SH_{n+1}^+(W) = \SH_{n}^-(W) = \Q , \\ \nonumber
& \SH_{q }^+ (W) = 0 \,\mbox{ if }\, q \geq n+2.
\end{align}
This proves the second assertion. 

Next consider the Gysin sequence
\begin{center}
\begin{tikzpicture}
\node (a1) {$\SH_*^{+, G}(W)$} ;
\node [right of = a1, xshift=30mm](a2) {$\SH_{*-2}^{+, G}(W)$} ;
\node [below of = a1, yshift=-12mm](b1) {};
\node [right of =b1, xshift=10mm](b2) {$\SH_*^{+}(W)$};
\path[draw, ->](a1) -- node[yshift=3mm]{$D$} (a2);
\path[draw, ->](b2) -- node[yshift=0mm, xshift=-4mm]{} (a1);
\path[draw, ->](a2) -- node[yshift=0mm, xshift=5mm]{$[+1]$} (b2);
\end{tikzpicture}
\end{center}
where $D$ is the shift operator. Then 
$$
\SH_{q+2}^{+,G}(W) \cong \SH_{q}^{+,G}(W) \,\,\,\mbox{ if }\,\,\, q
\geq n+1.
$$ 
Recall also from Section \ref{sec:sympl_hom} that $\SH(W) = 0$ if and
only if $\SH^G(W) = 0$.  From the long exact sequence, we see that
$$
\SH_{r+1}^{+,G}(W) = \SH_{r}^{-,G}(W)= \bigoplus_{p+q=r} \H_{p+n}(W, \p W)
\otimes \H_q(\C P^{\infty}).
$$
This isomorphism commutes with $D$ and, on the right, $D$ is given by
the pairing
$\H_q(\C P^{\infty})\to \H_{q-2}(\C P^{\infty})$ with a generator of
$\H^2(\C P^{\infty})$. This proves assertion (iii) and completes the
proof of the theorem.
\end{proof}

With this calculation in mind, we are in the position to define
(equivariant) homological symplectic capacities, aka spectral
invariants or action selectors, depending on the perspective. The
construction follows the standard path which goes back to
\cite{EH,HZ,Sc,Vi:gen}. (See also \cite{GG:convex,GH} for a recent
detailed treatment in the case where $W$ is a ball.)

To a non-zero class $\beta\in \SH^{+}(W)$, we associate the ``capacity''
$$
\s(\beta,W)=\inf\{a\in\R\mid \beta\in\im(i_a) \}\in\R,
$$
where the map $i_a\colon\SH^{+,(-\infty,a)}(W)\to \SH^{+}(W)$ is
induced by the inclusion of the complexes. (When $\beta=0$, we have,
by defintion,  $\s(\beta,W)=-\infty$.) This capacity can be
viewed as a function of $\beta$ or $W$. In the latter case,
$\s(\beta,W)$ has all expected features of a symplectic capacity as
long as $W$ varies within a suitably chosen class of manifolds with
naturally isomorphic homology groups $\SH^+(W)$. (We omit a
detailed and formal discussion of the general capacity properties of
$\s(\beta,W)$ and other capacities introduced below, for they are not
essential for our purposes.) For $\beta\in \SH^{+,G}(W)$, the
equivariant capacity $\s^G(\beta,W)$ is defined in a similar fashion.

By assertion (ii) of Proposition \ref{prop:hom-calc}, every class
$\zeta\in \H_*(W,\p W)$ gives rise to class $\zeta^+\in\SH^{+}(W)$ and
we set $\s_\zeta(W)=\s(\zeta^+,W)$. The capacity arising from the unit
$\zeta= [W,\p W]$ is of particular interest and we denote it by
$\s(W)$. Likewise, by \eqref{eq:tensor}, we can associate to $\zeta$ a
sequence of classes
$\zeta^G_k=\zeta^+\otimes\sigma_k\in \SH^{+,G}(W)$,
$k=0,\,1\,,2,\ldots$, where $\sigma_k$ is a generator in
$\H_{2k}(\CP^\infty)$ and $D(\zeta^G_{k+1})=\zeta^G_k$, and we set
$$
\s^G_{\zeta,k}(W):=\s(\zeta^G_k,W).
$$
When $\zeta=[W,\p W]$, we will simply write
$\s^G_k:=\s^G_{\zeta,k}$. The operator $D$ does not increase the
action filtration (see, e.g., \cite{BO:Gysin,GG:convex}), and hence
\begin{equation}
\label{eq:incr}
\s_{\zeta,0}^G(W)\leq \s_{\zeta,1}^G(W)\leq \s_{\zeta,2}^G(W)\leq\ldots.
\end{equation}

\begin{Lemma}
\label{lemma:pos}
The capacities are non-negative:
\begin{equation}
\label{eq:pos}
\s_\zeta(W)\geq 0 \textrm{ and } \s_{\zeta,k}^G(W)\geq 0
\end{equation}
and 
\begin{equation}
\label{eq:eqv-noneqv}
\s^G_{\zeta,0}(W)\leq \s_\zeta(W).
\end{equation}
\end{Lemma}

These inequalities are well known when $W$ is exact. (Moreover, then
all capacities are strictly positive.) However, when $W$ is only
assumed to be symplectically aspherical, non-trivial closed Reeb
orbits on $\p W$ can possibly have negative Hamiltonian action (i.e.,
symplectic area), and \eqref{eq:pos} is not entirely obvious.

\begin{proof}
  To prove \eqref{eq:pos} for, say, $\s_\zeta(W)$, consider an
  admissible Hamiltonian $H$, which is non-degenerate and bounded from
  below by $-\delta<0$ on $W$.  It is clear that the action selector
  corresponding to $\zeta^+$ for $H$ is also bounded from below by
  $-\delta$. Indeed, after a small non-degenerate perturbation of $H$
  outside $W$, the value of the selector is attained on an orbit which
  is connected by a Floer trajectory to a critical point of $H$ in
  $W$. Passing to the limit, we see that $\s_\zeta(W)\geq 0$. For the
  capacities $\s_{\zeta,k}^G$ the argument is similar.

  The proof of \eqref{eq:eqv-noneqv} is identical to the argument in
  the case where $W$ is exact. Namely, reasoning as in the proof of
  Proposition \ref{prop:hom-calc}, it is easy to show that the natural
  map
$$
\H_*(W,\p W)\to \SH^-(W)\stackrel{\cong}{\to} \SH^+(W)\to \SH^{+,G}(W),
$$
where we suppressed in the notation the grading shift by the second
arrow isomorphism, sends $\zeta$ to
$\zeta_0^G=\zeta^+\otimes \sigma_0$. With this in mind,
\eqref{eq:eqv-noneqv} follows from the commutative diagram 
\begin{center}
\begin{tikzpicture}
\node (a1) {$\SH^{+,(-\infty,a)}(W)$};
\node [right of=a1, xshift=25mm] (a2) {$\SH^{+}(W)$};
\node [below of=a1, yshift=-8mm] (b1) {$\SH^{+,(-\infty,a),S^1}(W)$};
\node [right of=b1, xshift=25mm] (b2) {$\SH^{+,S^1}(W)$};

\path[draw, -latex'](a1) -- node[yshift=3mm]{} (a2);
\path[draw, -latex'](b1) -- node[yshift=3mm]{} (b2);
\path[draw, -latex'](a1) -- node[xshift=-5mm, yshift=1mm]{} (b1);
\path[draw, -latex'](a2) -- node[yshift=1mm, xshift=2mm]{} (b2);
\end{tikzpicture}
\end{center}
\end{proof}


\begin{Remark}
  We expect that the strict inequalities also hold in
  \eqref{eq:pos}. However, proving this would require a more subtle
  argument. One could use, for instance, a continuation or ``transfer''
  map between $W$ and a slightly shrunk domain $W'$ to show that
  this map decreases the action by a certain amount and reasoning as
  in the proof of Theorem \ref{thm:van}. In Section
  \ref{sec:prequant}, we will show that the strict inequalities hold
  for prequantization bundles by a rather simple and different
  argument.
 \end{Remark}

 \begin{Remark} Proposition \ref{prop:hom-calc} readily extends to the
   setting where $W$ is monotone (or negative monotone, provided that
   the Floer homology is defined) once a Novikov ring is incorporated
   into the isomorphisms.
 \end{Remark}

\subsection{Uniform instability of the symplectic homology}
Let us now turn to quantitative consequences of vanishing of the
symplectic homology.

We say that the filtered symplectic homology of $W$ is \emph{uniformly
  unstable} if the natural ``quotient-inclusion'' map
\begin{equation}
\label{eq:van}
\SH^I(W)\to \SH^{I+c}(W)
\end{equation}
is zero for every interval $I$ (possibly infinite) and some constant
$c\geq 0$ independent of $I$. One way to interpret this defintion,
inspired by the results in \cite{Su}, is that every element of the
filtered homology is ``noise'' on the $c$-scale or, equivalently, that
all bars in the barcode associated with this homology have length no
longer than $c$. (See, e.g., \cite{PS,UZ} for a discussion of barcodes
and persistence modules in the context of symplectic topology.)

The requirement that the homology is uniformly unstable is seemingly
stronger than that the total homology vanishes: setting $I=\R$ we
conclude that $\SH(W)=0$.  However, as was pointed out to us by Kei
Irie, \cite{Ir}, the two conditions are equivalent for Liouville
domains. In other words, somewhat surprisingly, vanishing of the
total homology is equivalent to the uniform instability of the filtered
homology.  The next proposition is a minor generalization of this
observation.

\begin{Proposition}
\label{prop:irie}
Assume that
$\omega|_{\pi_2(W)}=0$. The following two conditions are equivalent:
\begin{itemize}
\item[\rm{(i)}] $\SH(W)=0$ and
\item[\rm{(ii)}] there exists a constant $c_0>0$ such that for any
  $c>c_0$ and any interval $I\subset\R$ the map \eqref{eq:van} is zero.
\end{itemize}
Moreover, the smallest constant $c_0$ with this property is exactly
the capacity $\s(W)$. 
\end{Proposition}
\begin{proof} 
  As has been pointed out above, to prove the implication \rm{(ii)}
  $\Rightarrow$ \rm{(i)}, it is enough to set $I= \R$ in \rm{(ii)}.
  Indeed, then \eqref{eq:van} is simultaneously zero and the identity
  map, which is only possible when $\SH(W)=0$.

  Let us prove the converse.  Assume that $\SH(W)=0$ and consider the
  natural map $\psi\colon \SH^-(W) \rightarrow \SH^{(-\infty, c)}(W)$.
  By definition, $\s(W)=\inf\{ c\mid\psi(\zeta)=0\}<\infty $ where we
  took $\zeta$ to be the image of the fundamental class
  $[W,\p W]$ in $\SH^-(W)$; see Proposition \ref{prop:hom-calc}.  Our
  goal is to show that the map \eqref{eq:van} vanishes for any
  $c>\s(W)$ and any interval $I$.

  {\bf Step 1.} For $a\notin \CS_\omega(W),$ consider an interval
  $I=(-\infty, a).$ We have the following commutative diagram where
  the horizontal maps are given by the pair-of-pants product:
\begin{center}
\begin{tikzpicture}

\node (a1) {$\SH^{(-\infty,a)}(W)\otimes\SH^{(-\infty,0]}(W)$};
\node [right of=a1, xshift=45mm] (a2) {$\SH^{(-\infty,a)}(W)$};
\node [below of=a1, yshift=-10mm] (b1) {$\SH^{(-\infty,a)}(W)\otimes\SH^{(-\infty,c)}(W)$};
\node [right of=b1, xshift=45mm] (b2) {$\SH^{(-\infty,a+c)}(W)$};

\path[draw, -latex'](a1) -- node[yshift=3mm]{} (a2);
\path[draw, -latex'](b1) -- node[yshift=3mm]{} (b2);
\path[draw, -latex'](a1) -- node[xshift=-5mm, yshift=1mm]{id$\otimes \psi$} (b1);
\path[draw, -latex'](a2) -- node[yshift=1mm, xshift=2mm]{$\phi$} (b2);

\end{tikzpicture}
\end{center}
Recall that $\zeta \in \SH^{-}(W)$ is a unit with respect to this
product.(We refer the reader to \cite{AS} for the definition of the
pair-of-paints product applicable in this case and also to, e.g.,
\cite{Ri:JT}.) Thus, for any $\sigma \in \SH^{(-\infty,a)}(W)$,
\begin{center}
\begin{tikzcd}
  \sigma\otimes \zeta \arrow[r, mapsto] \arrow[d, mapsto]
    & \sigma \arrow[d,  mapsto] \\
  \sigma\otimes 0 \arrow[r, mapsto]
&0 .\end{tikzcd}
\end{center}
Hence the map $\phi$ vanishes. 

{\bf Step 2.} For $a,b \notin \CS_\omega(W),$ consider an interval
$I=(a,b).$ We have the following commutative diagram:
\begin{center}
\begin{tikzpicture}

\node (a0) {};
\node [right of=a0, xshift=9mm] (a1) {$\SH_k^{(-\infty,a)}(W)$};
\node [right of=a1, xshift=25mm] (a2) {$\SH_k^{(-\infty,b)}(W)$};
\node [right of=a2, xshift=25mm] (a3) {$\SH_k^{(a,b)}(W)$};
\node [right of=a3, xshift=20mm] (a4) {};

\node [below of=a0, yshift= -7mm] (b0) {};
\node [right of=b0, xshift=9mm] (b1) {$\SH_k^{(-\infty,a+c)}(W)$};
\node [right of=b1, xshift=25mm] (b2) {$\SH_k^{(-\infty,b+c)}(W)$};
\node [right of=b2, xshift=25mm] (b3) {$\SH_k^{(a+c,b+c)}(W)$};
\node [right of=b3, xshift=20mm] (b4) {}; 

\path[draw, -latex'](a0) -- node[yshift=3mm]{} (a1);
\path[draw, -latex'](a1) -- node[yshift=3mm]{} (a2);
\path[draw, -latex'](a2) -- node[yshift=3mm]{} (a3);
\path[draw, -latex'](a3) -- node[yshift=3mm]{} (a4);

\path[draw, -latex'](b0) -- node[yshift=3mm]{} (b1);
\path[draw, -latex'](b1) -- node[yshift=3mm]{} (b2);
\path[draw, -latex'](b2) -- node[yshift=3mm]{} (b3);
\path[draw, -latex'](b3) -- node[yshift=3mm]{} (b4);

\path[draw, -latex'](a1) -- node[yshift=0mm, xshift=3mm]{$\phi_1$} (b1);
\path[draw, -latex'](a2) -- node[yshift=0mm, xshift=3mm]{$\phi_2$} (b2);
\path[draw, -latex'](a3) -- node[yshift=0mm, xshift=2mm]{$\psi$} (b3);

\end{tikzpicture}
\end{center}
By Step 1, the maps $\phi_1,\,\, \phi_2$ are zero maps. Hence the map $\psi$ vanishes. 

{\bf Step 3.} For $a\notin \CS_\omega(W)$ consider an interval
$I=(a,\infty).$ At Step 2, we obtained the zero map
$\psi\colon \SH^{(a,b)}(W) \rightarrow \SH^{(a+c,b+c)}(W)$. By taking
$b$ to $\infty,$ we see the map
$\psi\colon \SH^{(a,\infty)}(W) \rightarrow \SH^{(a+c,\infty)}(W)$
vanishes.
\end{proof}

\begin{Remark}
\label{rmk:vanish-equiv}
It is worth pointing out that Proposition
\ref{prop:irie} and Theorem \ref{thm:van} below do not have a
counterpart in the equivariant setting. Indeed,  when $W$ is the
standard symplectic ball $B^{2n}$ the maps
$$ 
\SH^{G, (a,\,\infty)}(W)\to \SH^{G, (a+c,\,\infty)}(W)
$$
are non-zero for any $a$ and $c\geq 0$ while $\SH^G(W)=0$.
\end{Remark}

\subsection{Growth of symplectic capacities}
\label{eq:cap-growth}

Another consequence of vanishing of the symplectic homology is an upper
bound on the growth of the equivariant symplectic capacities.

\begin{Proposition}
\label{prop:capacities}
Assume that $\omega|_{\pi_2(W)}=0$. Then, for every
$\zeta\in \H_d(W,\p W)$ and $k$ such that $2k\geq 2n-d$, we have
$$
0\leq \s_{\zeta,k+1}^G(W)-\s_{\zeta,k}^G(W)\leq \s(W).
$$
\end{Proposition}
\begin{proof} 
The first inequality is simply the assertion that the sequence
$\s_{\zeta,k}(W)$ is (non-strictly) monotone increasing (see
\eqref{eq:incr}) and, as has been pointed out in Section
\ref{sec:hom-cap}, this is a consequence of the fact that
the operator $D$ does not increase the
action filtration (see, e.g., \cite{BO:Gysin,GG:convex}).

Let us show that $\s_{\zeta,k+1}^G(W)-\s_{\zeta,k}^G(W)\leq \s(W)$.
By Proposition \ref{prop:hom-calc},
\begin{align*}
\SH_{n+1}^{+,G}(W) & \cong \bigoplus_{r=0}^n \H_{2n-2r}(W, \p W) , \\
\SH_{n+2}^{+,G}(W) & \cong \bigoplus_{r=1}^n \H_{2n-2r+1}(W, \p W).  
\end{align*}
For $k \geq 1$ and $e > \s(W),$ we have the following commutative diagram:

\begin{center}
\begin{tikzpicture}

\node (b0) {};
\node [right of=b0, xshift=9mm] (b1) {$\SH^{(b+e,\infty)}_{n+r}(W)$};
\node [right of=b1, xshift=25mm] (b2) {$\SH^{(b+e,\infty),G}_{n+r+2}(W)$};
\node [right of=b2, xshift=25mm] (b3) {$\SH^{(b+e,\infty),G}_{n+r}(W)$};
\node [right of=b3, xshift=20mm] (b4) {};

\node [below of=b0, yshift= -7mm] (c0) {};
\node [right of=c0, xshift=9mm] (c1) {$\SH^{(b,\infty)}_{n+r+2}(W)$};
\node [right of=c1, xshift=25mm] (c2) {$\SH^{(b,\infty),G}_{n+r+2}(W)$};
\node [right of=c2, xshift=25mm] (c3) {$\SH^{(b,\infty),G}_{n+r}(W)$};
\node [right of=c3, xshift=20mm] (c4) {}; 

 \node [below of=c0, yshift= -7mm] (d0) {};
 \node [right of=d0, xshift=9mm] (d1) {$\SH^{+}_{n+r+2}(W)$};
 \node [right of=d1, xshift=25mm] (d2) {$\SH^{+,G}_{n+r+2}(W)$};
 \node [right of=d2, xshift=25mm] (d3) {$\SH^{+,G}_{n+r}(W)$};
 \node [right of=d3, xshift=20mm] (d4) {$\SH^{+}_{n+r+1}(W)$};

\node [below of=d0, yshift= -7mm] (e0) {};
\node [right of=e0, xshift=9mm] (e1) {$\SH^{(\epsilon, b)}_{n+r+2}(W)$};
\node [right of=e1, xshift=25mm] (e2) {$\SH^{(\epsilon, b),G}_{n+r+2}(W)$};
\node [right of=e2, xshift=25mm] (e3) {$\SH^{(\epsilon, b),G}_{n+r}(W)$};
\node [right of=e3, xshift=20mm] (e4) {};

\node [below of=e0, yshift= -3mm] (f0) {};
\node [right of=f0, xshift=9mm] (f1) {};
\node [right of=f1, xshift=25mm] (f2) {};
\node [right of=f2, xshift=25mm] (f3) {};
\node [right of=f3, xshift=20mm] (f4) {};
 \path[draw, -latex'](b0) -- node[yshift=3mm]{} (b1);
 \path[draw, -latex'](b1) -- node[yshift=3mm]{} (b2);
 \path[draw, -latex'](b2) -- node[yshift=3mm]{} (b3);
 \path[draw, -latex'](b3) -- node[yshift=3mm]{} (b4);

 \path[draw, -latex'](c0) -- node[yshift=3mm]{} (c1);
 \path[draw, -latex'](c1) -- node[yshift=3mm]{$\pi_*$} (c2);
 \path[draw, -latex'](c2) -- node[yshift=3mm]{$D_1$} (c3);
 \path[draw, -latex'](c3) -- node[yshift=3mm]{} (c4);

 \path[draw, -latex'](d0) -- node[yshift=3mm]{} (d1);
 \path[draw, -latex'](d1) -- node[yshift=3mm]{} (d2);
 \path[draw, -latex'](d2) -- node[yshift=3mm]{$D_2$ } (d3);
 \path[draw, -latex'](d3) -- node[yshift=3mm]{} (d4);

 \path[draw, -latex'](e0) -- node[yshift=3mm]{} (e1);
 \path[draw, -latex'](e1) -- node[yshift=3mm]{} (e2);
 \path[draw, -latex'](e2) -- node[yshift=3mm]{} (e3);
 \path[draw, -latex'](e3) -- node[yshift=3mm]{} (e4);

  \path[draw, -latex'](c1) -- node[yshift=0mm, xshift=4mm]{\eqref{eq:van}} (b1);
 \path[draw, -latex'](c2) -- node[yshift=0mm, xshift=3mm]{$f$} (b2);
 \path[draw, -latex'](c3) -- node[yshift=0mm, xshift=5mm]{} (b3);
  
  \path[draw, -latex'](d1) -- node[yshift=0mm, xshift=3mm]{} (c1);
  \path[draw, -latex'](d2) -- node[yshift=0mm, xshift=4.5mm]{$j_{r+2}$} (c2);
  \path[draw, -latex'](d3) -- node[yshift=0mm, xshift=3.5mm]{$j_{r}$} (c3);
  
  \path[draw, -latex'](e1) -- node[yshift=0mm, xshift=3mm]{} (d1);
  \path[draw, -latex'](e2) -- node[yshift=0mm, xshift=4.5mm]{$i_{r+2}$} (d2);
  \path[draw, -latex'](e3) -- node[yshift=0mm, xshift=3.5mm]{$i_{r}$} (d3);

 \path[draw, -latex'](f1) -- node[yshift=0mm, xshift=3mm]{} (e1);
 \path[draw, -latex'](f2) -- node[yshift=0mm, xshift=3mm]{} (e2);
 \path[draw, -latex'](f3) -- node[yshift=0mm, xshift=5mm]{} (e3);    

\end{tikzpicture}
\end{center}
Except for the first row, each row is the Gysin sequence and $D_i$ is
the shift operator for $i=1,\, 2$; each column (except again the
entries in the first row) comes from the short exact sequence
\begin{center}
\begin{tikzpicture}
\node (a0) {$0$};
\node [right of=a0, xshift=12mm] (a1) {$\CF^{(\epsilon, b)}$};
\node [right of=a1, xshift=15mm] (a2) {$\CF^{(\epsilon, \infty)}$};
\node [right of=a2, xshift=15mm] (a3) {$\CF^{(b,\infty)}$};
\node [right of=a3, xshift=12mm] (a4) {$0,$};

 \path[draw, -latex'](a0) -- node[yshift=3mm]{} (a1);
 \path[draw, -latex'](a1) -- node[yshift=3mm]{} (a2);
 \path[draw, -latex'](a2) -- node[yshift=3mm]{} (a3);
 \path[draw, -latex'](a3) -- node[yshift=3mm]{} (a4);
 
\end{tikzpicture}
\end{center}
where $\CF^I$ is a filtered Floer complex; the first row is obtained
by shifting the action interval $(b,\infty)$ by $e$ upward.

Consider a class $\zeta\in \H_d(W,\p W)$. Since $2k\geq 2n-d$, the
class $\zeta^G_k$ lies in $\SH_{n+r}^{+,G}(W)$ for some $r\geq
1$. From \eqref{zero}, we see that $\SH^{+}_{n+r+2}(W)=0$ and
$\SH^{+}_{n+r+1}(W)=0$ for all $r\geq 1$. Hence the map $D_2$ is an
isomorphism. Let $\xi$ be the preimage of $\zeta_k^G$ under
$D_2$. Assume that there exists a class
$\zeta' \in \SH_{n+r}^{(\epsilon,b),G}(W)$ which is sent to
$\zeta_k^G$ by the map $i_{r}$. Then we see that
$(i_r\circ j_r )(\zeta') = 0$. By commutativity of the diagram,
$(D_1\circ j_{r+2})(\xi) = 0.$ Hence, there exists a class
$\xi' \in \SH_{n+r+2}^{(b,\infty)}(W)$ such that
$\pi_*(\xi') = j_{r+2}(\xi).$ Again, by commutativity of the diagram,
$(f\circ j_{r+2})(\xi) = 0.$ Hence, $\s_{\zeta,k+1}^G(W) \leq b+e. $
\end{proof}

\subsection{Vanishing and displacement}
A geometrical counterpart of the condition that $\SH(W)=0$ is the
requirement that $W$ is (stably) displaceable in $\wh{W}$. In this
section, we will revisit and generalize 
the well-known fact that $\SH(W)=0$ for displaceable Liouville domains
$W$. In particular, we extend this result to monotone or
negative monotone symplectic manifolds.

\begin{Theorem}
\label{thm:van}
Assume that $W$ is positive or negative monotone and that $W$ is
displaceable in $\wh{W}$ with displacement energy $e(W)$. Then, for
any $c>e(W)$ and any interval $I\subset \R$, the quotient-inclusion
map \eqref{eq:van} is zero. Thus the filtered symplectic homology is
uniformly unstable and, in particular, $\SH(W)=0$.
\end{Theorem}

This theorem generalizes the result that $\SH(W)=0$ for Liouville
domains displaceable in $\wh{W}$ proved in \cite{CF,CFO} via vanishing
of the Rabinowitz Floer homology. (See also \cite{Vi:GAFA} for the first
results in this direction.)  The proof of Theorem \ref{thm:van} when
$\omega|_{\pi_2(W)}=0$ is implicitly contained in \cite{Su}.  Thus the
main new point here is that this condition can be relaxed as that $W$ is
allowed to be positive or negative monotone. Note also that when $W$
is symplectically aspherical one can obtain the uniform instability as
a consequence of Proposition \ref{prop:irie} and of vanishing of the
homology although with a possibly different lower bound on $c$, which
turns out to be better. Namely, combining this proposition with
Theorem \ref{thm:van}, and also using Proposition
\ref{prop:capacities}, we have the following:

\begin{Corollary}
  Assume that $W$ is symplectically aspherical and displaceable in
  $\wh{W}$ with displacement energy $e(W)$. Then
$$
\s(W)\leq e(W)
$$
and thus, for $\zeta\in \H_d(W,\p W)$ and $2k\geq 2n-d$,
$$
0\leq \s_{\zeta,k+1}^G(W)-\s_{\zeta,k}^G(W)\leq e(W).
$$
\end{Corollary}

Furthermore, recall that $W$ is called \emph{stably displaceable} when
$W\times S^1$ is displaceable in $\wh{W}\times T^*S^1$ and the
displacement energy of $W\times S^1$ is then referred to as the
stable displacement energy $e_\st(W)$ of $W$.  Combining the K\"unneth
formula from \cite{Oa:Kunneth} with Theorem \ref{thm:van}, we obtain
the following well-known result.

\begin{Corollary}
\label{cor:van}
Assume that $W$ is a Liouville domain and stably displaceable in
$\wh{W}$ with stable displacement energy $e_\st(W)$.  Then for any
$c>e_\st(W)$ and any interval $I\subset \R$ the map \eqref{eq:van} is
zero. In particular, $\SH(W)=0$.
\end{Corollary}

\begin{proof}[Proof of Theorem \ref{thm:van}]
  First, assume that $W$ is aspherical and $I=[a,b]$ for
  $a,b\notin\CS(\alpha).$ Consider a cofinal sequence of admissible
  Hamiltonians $\{H_i \colon S^1 \times \wh{W} \to \R\}$ satisfying
  the following conditions:
\begin{itemize}
\item[](i) $H_i$ is $C^2$-small on $W$;
\item[](ii) 
  $H_i = h_i(r)$ on the cylindrical part $\p W \times [1,\infty)$,
  where $h_i$ is an increasing function such that, for some $r_i > 1$,
  $h''_i \geq 0$ on $[1, r_i]$ and $h_i(r) = k_i r + l_i$ with
  $ k_i \notin \CS(\alpha)$ for $ r \in [r_i,\infty)$;
\item[](iii) $k_i \to\infty$ and $r_i \to 1$ as $i\to\infty$.
\end{itemize}
Define a sequence of Hamiltonians $\{F_i \colon S^1 \times \wh{W} \to \R\}$ by 
\begin{itemize}
\item[](i) $F_i$ is on $C^2$-small on $W$;
\item[](ii) $F_i = f_i(r)$ on the cylindrical part
  $\p W \times [1,\infty)$, where for $\epsilon>0$
\[
  f_i(r)=
\begin{cases}
h_i(r) & \text{ if }   r \in [1,r_i] \\ 
k_i^-(r-r_i) & \text{ if }  r \in [r_i, r_i^- - \epsilon]\\ 
c_i & \text{ if }  r \in [r_i^-, r_i^+] \\
k_i^+(r-r_i) & \text{ if }  r \in [r_i^+ + \epsilon, \infty)
\end{cases}
\]
for  $k_i^+ = k_i$, $ k_i^-\notin \CS(\alpha)$ and $k_i^- > k_i^+$;
\item[](iii)
  $f''_i(r) \leq 0 \text{ if } r \in [r_i^- - \epsilon, r_i^-]$
  and
  $f''_i(r) \geq 0 \text{ if } r \in [r_i^+, r_i^+ + \epsilon]$;
\item[](iv) $\min F_i \to 0$ as $i\to\infty$.
\end{itemize}
The graphs of $F_i$ and $H_i$ are shown in Fig.\ \ref{fig:Graph01}.
\begin{figure}[h]
\centering
\includegraphics[scale=0.34]{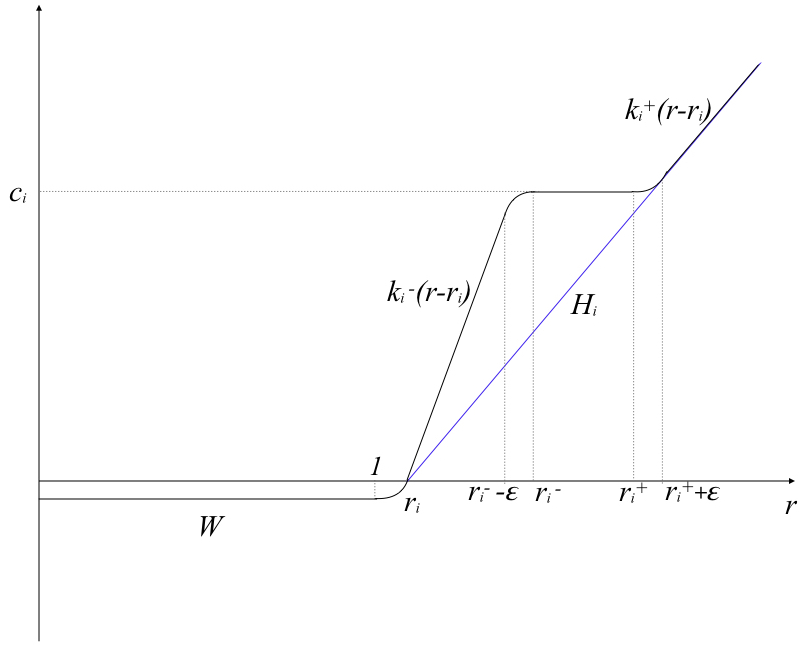}
\caption{Functions $F_i$ and $H_i$}
\label{fig:Graph01}
\end{figure} 

Then the sequence $\{F_i \}$ is cofinal and $F_i\geq H_i$. 
Thus we have  
$$
\SH^I(W) = \Dlim \HF^I(F_i),
$$ 
where the limit is taken over the sequence $\{F_i \}$.

Let us show that the map $\HF^I(F)\to \HF^{I+e}(F)$ is zero for a
Hamiltonian $F\in\{F_i\}_{i\geq N}$ and sufficiently large $N$.  By
the assumption that $W$ is displaceable in $\wh{W}$, there exists a
Hamiltonian $K\colon S^1 \times \wh{W}\to\R$ such that
$\phi_K^1(W) \cap W =\emptyset$ and $e(W) < \|K\| < e,$ where
$\phi_K^t$ is the Hamiltonian flow of $K.$ Consider the positive and
negative parts of Hofer's norm of $K$:
$$
\|K\|_+ = \int_{S^1}\,\, \max_{x\in\wh{W}} K(t,x) \,dt
$$
and
$$
\|K\|_- = \int_{S^1}\,\, -\min_{x\in\wh{W}} K(t,x) \,dt.
$$ 
Then $\|K\| = \|K\|_+ + \|K\|_-$. Choose a constant $s \gg 0$ 
meeting the following conditions: 
\begin{align*}
&\inf\CS(H) + s > b + \|K\|_+ ,\\ \nonumber
&\inf\CS(K) + s > b + \|K\|_+  . \nonumber
\end{align*}
Select constants $c$ and $r_{\pm}$ such that 
\begin{align*}
& c > s ; \\ \nonumber
& \supp K \subset W \cup \p W \times [1, r^+] , \\ \nonumber
& \phi_K^1 \mbox{ displaces } W \cup \p W \times [1, r^-] .
\end{align*}
\bigskip
Define a Hamiltonian $K\#F$ by 
$$K\#F(t,x) = K(t,x) + F\left(t, \left(\phi_K^t \right)^{-1}(x)\right).$$
Then $\phi_{K\#F}^t = \phi_K^t \circ \phi_F^t$ 
which is homotopic to the catenation of $\phi_F^t$ with $\phi_K^t$. 

Let $\PP(K)$ be the collection of one-periodic orbits of $K$ which are
contractible in $W$.  The collections $\PP(F)$ and $\PP(K\#F)$ are
defined similarly.  Then there is a one-to-one correspondence between
$\PP(H)$ and $\PP(F)$ for the orbits lying on a level
$r\in [r^+, r^+ + \epsilon]$.  Denote by $\PP(F, r^+)$ the collection
of such orbits in $\PP(F)$.  It is not hard to see that $\PP(K\#F)$
consists of the orbits in $\PP(K)$ and the orbits in $\PP(F, r^+).$
Indeed, all of the orbits in $\PP(F)$ on $r \leq r^-$ are displaced by
$\phi_K^1$.  Thus, the orbits near $r=r^+$ survive the displacement by
$\phi_K^1.$

Evaluating the action functional for  $x \in \PP(K)$ and $y \in
\PP(F,r^+)$, we have 
\begin{align*}
\CA_{K\#F}(x) & = -\int_{\bar{x}} \hat{\omega} + \int_{S^1}\, K\#F(t,x(t)) \,dt \\
&=\CA_K(x) +   \int_{S^1}\, F\left(t,\left(\phi_K^t\right)^{-1}(x(t))\right) \,dt \\
& = \CA_K(x) + c \\
&\geq b+\|K\|_+,
\end{align*}
where $\bar{x}$ is a capping of $x.$

Let $z \in \PP(H)$ be the orbit corresponding to $y \in \PP(F,r^+)$. Then
\begin{align*}
\CA_{K\#F}(y) & = -\int_{\bar{y}} \hat{\omega} + \int_{S^1}\, K\#F(t,y(t)) \,dt \\
& = -\int_{\bar{z}} \hat{\omega} +  
(r^+-1)\int_{z} \alpha  + \int_{S^1}\, H(t,z(t)) \,dt  + c\\
& = \CA_{H}(z) +  (r^+-1)\int_{z} \alpha + c \\
& \geq b+\|K\|_+,
\end{align*}
where $\bar{y}$ and $\bar{z}$ are cappings of $y$ and $z$, respectively. 

Let $H_s$ be a linear homotopy from $F$ to $K\# F$.  For
$x \in \PP(F)$ and $y \in \PP(K\# F),$ consider the moduli space
\begin{align*}
\Mm(x,y,H_s,J_s)=\{ u \in C^{\infty}(S^1\times \R, \wh{W}) &
| \lim_{s\to-\infty}u(t,s) = x, \lim_{s\to+\infty}u(t,s)=y, \\
& \p_s u + J_s (\p_t u -X_{H_s}(u) )  = 0 \}.
\end{align*}
If the moduli space is not empty, 
\begin{align*}
\CA_{K\#F}(y) &\leq \CA_F(x) + \int_{S^1} \int_{\R}\frac{\p H_s}{\p s} (u)\,ds dt \\
                       &\leq \CA_F(x) + \int_{S^1} \max_{x\in \wh{W}} (K\#F - F)\, dt  \\
                       &= \CA_F(x) + \|K\|_+ .
\end{align*}
Similarly, consider a linear homotopy from $K\#F$ to $F$. Then for
$x \in \PP(F)$ and $y \in \PP(K\# F),$ we have
\begin{align*}
\CA_F(x) \leq \CA_{K\#F}(y) + \|K\|_- .
\end{align*}
For $e > \|K\|,$ we have the following commutative diagram:
\begin{center}
\begin{tikzpicture}
\node (a1) {$\HF^{I}(F)$};
\node [right of=a1, xshift=60mm] (a2) {$\HF^{I+e}(F)$};
\node [below of = a1, yshift= -10mm](b1) {$\HF^{I+\|K\|_+}(K\#F)$};
\node [right of=b1, xshift=60mm] (b2) {$\HF^{I+\|K\|_+ +\|K\|_-}(F)$}; 
\path[draw, -latex'](a1) -- node[yshift=3mm]{(1)} (a2);
\path[draw, -latex'](b1) -- node[yshift=3mm]{} (b2);
\path[draw, -latex'](a1) -- node[yshift=0mm, xshift=3mm]{(2)} (b1);
\path[draw, -latex'](b2) -- node[yshift=0mm, xshift=5mm]{} (a2);
\end{tikzpicture}
\end{center}
Since $\CA_{K\#F}(x) \geq b+\|K\|_+$ for every $x \in \PP(K\#F),$ the
map (2) vanishes.  The map (1) vanishes as well.  By taking direct
limit of the map (1) over the cofinal sequence $\{F_i\},$ we see that
the map \eqref{eq:van} vanishes.

Next let us show that the map \eqref{eq:van} also vanishes when $W$ is
monotone, i.e., $[\omega]|_{\pi_2(W)} = \lambda \s_1(TW)|_{\pi_2(W)}$
for some nonzero constant $\lambda$. Let $H$ be an admissible
Hamiltonian. Consider the set
$$
\CS_q(H)=\left\{\CA_H(\bx) \mid \Delta_H(\bx) \in [q,q+2n] \right\},
$$ 
where $\dim W = 2n$ and $\Delta_H(\bx)$ is the mean index
of $\bx$,
and the set $ \CS_q(K)$ defined similarly for the Hamiltonian
$K$. Clearly, these sets are compact.
Now, choose a constant $s \gg 0$ meeting the following conditions: 
\begin{align*}
&\inf\CS_q(H) + s > b + \|K\|_+ ;\\ \nonumber
&\inf\CS_q(K) + s > b + \|K\|_+ .  \nonumber
\end{align*}
By the same argument as the case where $W$ is aspherical, we conclude
that the map \eqref{eq:van} is zero.
\end{proof}


\section{Prequantization bundles}
\label{sec:prequant}

\subsection{Generalities}
Let $(B^{2m},\sigma)$ be a symplectically aspherical manifold such
that $\sigma$ is integral. Then there exists a principal $S^1$-bundle
$\pi\colon M\to B$ and an $S^1$-invariant one-form $\alpha_0$ on it (a
connection form with curvature $\sigma$) such that
$\pi^*\sigma=d\alpha_0$. This is a \emph{prequantization
  $S^1$-bundle}. The form $\alpha_0$ is automatically contact and the
contact structure $\ker\alpha_0$ is a connection on $\pi$. The Reeb flow
of $\alpha_0$ is (up to a factor) the $S^1$-action on $M$. The factor is
the integral of $\alpha_0$ over the fiber and its value depends on
conventions and essentially boils down to the definition of an
integral form and an integral de Rham class; see, e.g., \cite[App.\
A]{GGK}. Here we assume that this integral is $\pi$, i.e., $\sigma$ is
integral if and only if $[\sigma]\in \H^2(B;\pi\Z)$. (We use the same
notation $\pi$ for the number $3.14\ldots$ and for the projection of a
principle $S^1$-bundle to the base; what is what should be clear from
the context.) 
The associated line bundle $\pi\colon E\to B$ is a symplectic manifold
with symplectic form
$$
\omega=\frac{1}{2}\big(\pi^*\sigma+d(r^2\alpha_0)\big)
$$
where $r\colon E\to [0,\infty)$ is the fiberwise distance to the zero
section. We call $(E,\omega)$ a \emph{prequantization line
  bundle}. (In the context of algebraic geometry prequantization line
bundles are often referred to as negative line bundles.)  Note that
away from the zero section we have
$$
\omega=\frac{1}{2}d\big((1+r^2)\alpha_0\big).
$$

Let $\alpha=f\alpha_0$ be a contact form on $M$ supporting
$\ker\alpha_0$.  Without loss of generality we may assume that
$f>1/2$. Then the fiberwise star shaped hypersurface given by the
condition $(1+r^2)/2=f$ and denoted by $M_f$ or $M_\alpha$ has contact
type, and the restriction of the primitive $(1+r^2)\alpha_0/2$ to
$M_f$ is exactly $\alpha$.  The domain bounded by $M_f$ in $E$, which
is sometimes denoted by $W_f$ or $W_\alpha$ in what follows, is a
strong symplectic filling of $M_f$ diffeomorphic to the associated
disk bundle. Clearly, we can identify $E$ with the completion
$\widehat{W}$.

Denote by $\tpi_1(M)$ the collection of free homotopy
classes of loops in $M$ or equivalently the set of conjugacy classes
in $\pi_1(M)$. Furthermore, let $\ff$ be the free homotopy class of the
fiber in $M$ or in $E\setminus B$ and $\ff^\Z=\{\ff^k\mid k\in\Z\}$.

Since $\sigma$ is aspherical, the homotopy long exact sequence of the
$S^1$-bundle $M\to B$ splits and we have
$$
1\to\pi_1(S^1)\to \pi_1(M)\to \pi_1(B)\to 1.
$$
It is not hard to see that $\ff^\Z$ is the image of
$\Z\cong \pi_1(S^1)$ in $\tpi_1(M)=\tpi_1(E\setminus B)$; \cite[Lemma
4.1]{GGM:CC}. Furthermore, this is exactly the set of free homotopy
classes of loops $x$ with contractible projections to $B$, i.e., of
loops contractible in $E$. The one-to-one correspondence
$\ff^\Z\to \Z$ is given by the \emph{linking number} $L_B(x)$ of $x$
with $B$. This is simply the intersection index of a generic disk
bounded by $x$ with $B$.

In what follows, \emph{all loops and periodic orbits are assumed to be
  contractible in $E$} unless stated otherwise.

Recall that to every such loop $x$ in $M_\alpha$ we can associate two
actions: the \emph{symplectic action} $\CA_\omega(x)$ obtained by
integrating $\omega$ over a disk bounded by $x$ in $E$ or $W$ and the
\emph{contact action} $\CA_\alpha(x)$ which is the integral of
$\alpha$ over $x$. These two actions are in general different.

\begin{Example}
\label{exam:fibers}
Let $M$ be the $S^1$-bundle $r=\eps$ in $E$. Then the closed Reeb
orbits $x$ in $M$ are the iterated fibers. (In particular, every Reeb
orbit is closed.) As a straightforward calculation shows, we have
$\CA_\omega(x)=\pi k \eps^2/2$ and $\CA_\alpha(x)=\pi k (1+\eps^2)/2$
for $x\in \ff^k$.
\end{Example}

\begin{Lemma}
\label{lemma:action-diff}
Let $x$ be a loop in $M_\alpha$ in the free homotopy class $\ff^k$,
$k\in\Z$, i.e., $L_B(x)=k$. Then
\begin{equation}
\label{eq:action-diff}
\CA_\omega(x)=\CA_\alpha(x)-\frac{\pi}{2}k.
\end{equation}
\end{Lemma}

\begin{proof}
  It is clear that the difference $\CA_\omega(x)-\CA_\alpha(x)$ is a
  purely topological invariant completely determined by the free
  homotopy class of $x$ in $E\setminus B$. By Example
  \ref{exam:fibers}, we see that the difference is equal to $-\pi k/2$.
\end{proof}

It follows from Lemma \ref{lemma:action-diff} that the Hamiltonian
action is bounded from below by $-\pi k/2$ and the question if it can
really be negative was raised in, e.g., \cite{Oa:Leray-Serre}. Our next
proposition gives an affirmative answer to this question.

\begin{Proposition}
\label{prop:negative}
For every $k$ there exists a contact form $\alpha=f\alpha_0$, where
$f>1/2$, with a closed Reeb orbit $x$ in the class $\ff^k$ such that
$\CA_\omega(x)$ is arbitrarily close to $-\pi k/2$.
\end{Proposition}

\begin{proof}
  Recall that every free homotopy class can be realized by an embedded
  smooth oriented loop which is tangent to the contact structure; see,
  e.g., \cite{EM} and references therein. Let
  $y$ be such a loop in the class $\ff^k$. By moving $y$ slightly in
  the normal direction to $\dot{y}$ in $\xi=\ker\alpha_0$, i.e., in
  the direction of $J\dot{x}$ where $J$ is an almost complex structure
  on $\xi$ compatible with $d\alpha_0$, we obtain a transverse
  embedded loop $x$. The loop $x$ is nearly tangent to $\xi$ and we
  can ensure that
$$
0<\int_x \alpha_0 <\eps
$$
for an arbitrarily small $\eps>0$.

It is a standard (and easy to prove) fact that there exists a contact
form $\beta$ on $M$ supporting $\xi$ such that $x$ is, up to a
parametrization, a closed Reeb orbit of $\beta$. Let
$g=\alpha_0/\beta$, i.e., $g$ is defined by $\alpha_0=g\beta$. By
scaling $\beta$ if necessary, we can ensure that $g\leq 1$. In other
words, $1/g\geq 1$. It is not hard to see that $g|_x$ extends to
a function $h$ on $M$ so that $h/g> 1/2$ and the
derivative of $h$ in the normal direction to $x$ is zero, i.e.,
$\ker dh\supset \xi$ at all points of $x$. 

The latter condition guarantees that $x$ is still a closed Reeb orbit
of $\alpha:=h\beta=f\alpha_0$, where $f=h/g$. By construction,
$f>1/2$.  Furthermore, we have $\alpha|_x=\alpha_0|_x$, and hence
$$
0<\int_x \alpha = \int_x \alpha_0 <\eps.
$$
Thus, by \eqref{eq:action-diff}, 
$$
-\pi k/2<\CA_\omega(x)\leq\eps-\pi k/2
$$
and $\CA_\omega(x)$ can be made arbitrarily close to $-\pi k/2$.
\end{proof}

\subsection{Applications of homology vanishing to prequantization bundles}
 
When $B$ is aspherical, $\SH(W)=0$ for $W=W_f$ by the K\"unneth
formula from \cite{Oa:Leray-Serre}, and the results from Section
\ref{sec:vanishing} directly apply to $W$.

For instance, combining the Thom isomorphism $\H_*(B)=\H_*(W,\p W)[-2]$
with Proposition \ref{prop:hom-calc}, we obtain

\begin{Corollary}
\label{cor:hom-calc}
Assume that $B^{2m}$ is symplectically aspherical and $W=W_f$. Then
we have natural isomorphisms
\begin{itemize}
\item[\rm{(i)}] $\SH^-(W)=\H_*(B)[-m+1]$ and 
$\SH^{-,G}(W)=\H_*(B)\otimes \H_*(\CP^\infty)[-m+1]$;
\item[\rm{(ii)}] $\SH^+(W)=\H_*(B)[-m+2]$ and 
\begin{equation}
\label{eq:cor-tensor}
\SH^{+,G}(W)=\H_*(B)\otimes \H_*(\CP^\infty)[-m+2];
\end{equation}
\item[\rm{(iii)}] combined with the identification
  \eqref{eq:cor-tensor}, the Gysin sequence shift map
$$
\SH^{+,G}_{r+2}(W)\stackrel{D}{\longrightarrow}\SH^{+,G}_{r}(W)
$$
is the identity on the first factor and the map
$\H_{q+2}(\CP^\infty)\to \H_{q}(\CP^\infty)$, given by the pairing
with a suitably chosen generator of $\H^{2}(\CP^\infty)$, on the
second. In particular, $D$ is an isomorphism when $q\geq m+2$.
\end{itemize}
\end{Corollary}

\begin{Remark}
  Recall that even when $B$ is not aspherical but simply meets the
  standard conditions sufficient to have the (equivariant) symplectic
  homology of $W_f$ defined (e.g., that $E$ is weakly monotone), this
  homology is independent of $f$; see \cite{Vi:GAFA} and also
  \cite{Ri:AM} and Section \ref{sec:sympl_hom}.
\end{Remark}

Likewise, and Lemma \ref{lemma:pos} and Proposition
\ref{prop:capacities} yield

\begin{Corollary}
\label{cor:capacities}
Assume that
  $B$ is symplectically aspherical. Then, for every $\zeta\in \H_d(B^{2m})$
  and $W=W_f$, we
  have
\begin{equation}
\label{eq:cor-capacities}
0\leq \s^G_{\zeta,0}(W)\leq \s^G_{\zeta,1}(W)\leq
\s^G_{\zeta,2}(W)\leq\ldots
\textrm{ and }
\s^G_{\zeta,0}(W) \leq \s_\zeta(W),
\end{equation}
and, when $2k\geq 2m-d$,
$$
0\leq \s_{\zeta,k+1}^G(W)-\s_{\zeta,k}^G(W)\leq \s(W).
$$
Moreover, $\s^G_{\zeta,0}(W)>0$, and hence all capacities are strictly
positive.
\end{Corollary}

The new point here, when compared to the general results, is the last
assertion that the capacities are strictly positive. To see this, note
first that these capacities are monotone (with respect to inclusion) on
the domains $W_f$. Thus it suffices to show that $\s^G_{\zeta,0}(U)>0$
for a small tubular neighborhood $U$ of $B$ in $E$ bounded by the
$S^1$-bundle $r=\eps$. It is not hard to see that in this case
$\s^G_{\zeta,0}(U)=\CA_\omega(x)$ for a closed Reeb orbit $x$ on $M$;
see Section \ref{sec:linking-calc}. Hence, by Example
\ref{exam:fibers}, $\s^G_{\zeta,0}(U)\geq \pi\eps^2/2>0$.

\subsection{Stable displacement}
\label{sec:st-displ}
The zero section of a prequantization bundle $E$ is never
topologically displaceable since its intersection product with itself
is Poincar\'e dual, up to a non-zero factor, to $[\sigma]\neq 0$. As a
consequence, no compact subset containing the zero section is
topologically displaceable either. However, the situation changes
dramatically when one considers stable displaceability. (Recall that
$K\subset W$ is stably displaceable if $K\times S^1$ is displaceable
in $W\times T^*S^1$). The following observation essentially goes back
to \cite{Gu}.

\begin{Proposition}
\label{prop:st-displ}
The zero section is stably displaceable in $E$.
\end{Proposition}

\begin{proof}
  The zero section $B$ is a symplectic submanifold of $E$. Thus
  $B':=B\times S^1$ is nowhere coisotropic in $E\times T^*S^1$, i.e.,
  at no point the tangent space to $B'$ is coisotropic. Furthermore,
  $B'$ is smoothly infinitesimally displaceable: there exists a
  non-vanishing vector field along $B'$ which is nowhere tangent to
  $B'$. Now the proposition follows from \cite[Thm.\ 1.1]{Gu}. (When
  $\dim B=2$, one can also use the results from \cite{LS,Po:displ}).
\end{proof}

\begin{Remark}
\label{rmk:generalization}
This argument shows that every closed symplectic submanifold $B$ of
any symplectic manifold is stably displaceable.
\end{Remark}

As a consequence of Proposition \ref{prop:st-displ}, a sufficiently
small tubular neighborhood of $B$ is also stably displaceable in
$E$. However, in contrast with the case of Liouville manifolds, this
is not enough to conclude that arbitrary large tubular neighborhoods,
and thus all compact subsets of $E$, are also stably displaceable and
in fact they need not be.

\begin{Example}
  Let $E$ be the tautological bundle over $B=\CP^1$, i.e., a blow up of
  $\C^2$. Then $E$ contains a monotone torus $L$ which is the restriction of
  the $S^1$-bundle (for a suitable radius) to the equator. It is
  known that $\HF(L,L)\neq 0$; \cite[Sect.\ 4.4]{Sm}. Then, by the
  K\"unneth formula, $\HF(L',L')\neq 0$ where $L'=L\times S^1$ in
  $E\times T^*S^1$. See \cite{RS,Ve} for generalizations of this
  example and its connections with non-vanishing of symplectic homology.
\end{Example}

\begin{Remark}
  Note that Proposition \ref{prop:st-displ} holds for any base $B$,
  but gives no information about $\SH(W)$.  The reason is that the
  K\"unneth formula does not directly apply in this case even when $B$
  is aspherical; see \cite{Oa:Kunneth}.  The boundary of a tubular
  neighborhood $U$ of $B\times S^1$ in $E\times T^*S^1$ does not have
  contact type. Moreover, the symplectic form is not even exact near
  the boundary. As a consequence, the symplectic homology of $U$ is
  not defined. One can still introduce an \emph{ad hoc} variant of
  such a homology group to have the K\"unneth formula and then reason
  along the lines of the proof of Theorem \ref{thm:van} to show that
  this homology, and hence $\SH(W)$, vanishes. However, this argument
  is not much simpler than the proof in \cite{Oa:Leray-Serre}. (We
  refer the reader to \cite{Ri:AM,RS,Ve} for a variety of
  vanishing/non-vanishing results for $\SH(W)$ expressed in terms of
  the conditions on the base $B$.)

  On the other hand, the proposition does imply, via the K\"unneth
  formula, vanishing of the Rabinowitz Floer homology for low energy
  levels in $E$, i.e., for $r$ close to zero, proved originally in
  \cite{AK}. Note in this connection that, as was pointed out to us by
  Alex Oancea, the Rabinowitz Floer homology might depend on the
  energy level in this case.
\end{Remark}

Proposition \ref{prop:st-displ} has some standard consequences along
the lines of the almost existence theorem, the Weinstein conjecture
and lower bounds on the growth of periodic orbits, which all are
proved via variants of the displacement energy--capacity inequalities
and are accessible by several methods requiring somewhat different
assumptions on $(B,\sigma)$; see, e.g., \cite{Po:Geom} and also
\cite{Gi:We}. Here, dealing with the almost existence, we adopt the
setting from \cite{Sc} which is immediately applicable.

The condition which $E$ must satisfy then is that it is \emph{stably
  strongly semi-positive} in the sense of \cite{Sc}, which is the case
if and only if $N_B\geq n+1$ or $(B^{2n},\sigma)$ is positive
monotone, i.e., $c_1(TB)=\lambda[\omega]$ on $\pi_2(B)$ where
$\lambda\geq 0$ and in addition we require that
$\left<[\omega],\pi_2(B)\right>=0$ whenever
$\left<c_1(TB),\pi_2(B)\right>=0$. (Here $N_B$ is the minimal Chern
number of $B$. Note that $N_E=N_B-1$.)

\begin{Corollary}[Almost Existence in $E$ near $B$; \cite{Lu}]
\label{cor:OE}
Assume that $B$ is as above. Let $H$ be a smooth, proper, automous
Hamiltonian on $E$ and let $I$ be a (possibly empty) interval such
that $\{H=c\}$ is contained in a sufficiently small neighborhood of
$B$ in $E$. Then, for almost all $c\in I$ in the sense of measure
theory, the level $\{H=c\}$ carries a periodic orbit of $H$.
\end{Corollary}

This is an immediate consequence of the displacement energy--capacity
inequality from \cite{Sc}. (Note that in this case the level $\{H=c\}$
automatically bounds a domain in a small tubular neighborhood of $B$,
and hence the Hofer--Zehnder capacity of this domain, as a function of
$c$ with finite values, is defined; see \cite{MS} for a different
approach.) The corollary is not the most general result of this
kind. It is a particular case of the main theorem from
\cite{Lu}. However, our proof is simpler than the argument \emph{ibid}
and, in fact, Remark \ref{rmk:generalization} can be used to simplify
some parts of that argument.  Corollary \ref{cor:OE} implies the
Weinstein conjecture for contact type hypersurfaces in $E$ near $B$.
There are, of course, many other instances where the Weinstein
conjecture is known to hold for hypersurfaces in $E$. For example,
although to the best of knowledge it is still unknown if it holds in
general for prequantization bundles, it does hold under suitable
additional conditions. For instance, this is the case when $\sigma$ is
aspherical, \cite{Oa:Leray-Serre}, or more generally if
$\pi^*[\sigma]$ is nilpotent in the quantum cohomology of
$E$;~\cite{Ri:AM}.

The second application is along the lines of the Conley conjecture
(see \cite{GG:CC} or \cite[Prop.\ 4.13]{Vi:gen}) and concerns the
number or the growth of simple periodic orbits of compactly supported
Hamiltonian diffeomorphisms. For the sake of simplicity, we assume
that $\sigma$ is aspherical although this condition can be relaxed.

Let $H\colon S^1\times E\to \R$ be a compactly supported Hamiltonian.

\begin{Corollary}
  Assume that $\sigma$ is symplectically aspherical and $\supp H$ is
  contained in a sufficiently small neighborhood of $B\times S^1$ in
  $E\times T^*S^1$. Then $\varphi_H$ has infinitely many simple
  contractible periodic orbits with non-zero action provided that
  $\varphi_H\neq \id$. Moreover, when $H\geq 0$ the number of such
  orbits of period up to $k$ and with positive action grows at least
  linearly with $k$ unless, of course, $H=0$.
\end{Corollary}

Here the first assertion follows readily from \cite[Thm. 2]{FS} and
the second assertion is a consequence of \cite[Thm.\ 1.2]{Gu}.

\section{Linking number filtration}
\label{sec:linking}
In this section we construct and utilize a new filtration on the
positive (equivariant) symplectic homology of a prequantization disk
bundle $W\to B$. This filtration is, roughly speaking, given by the
linking number of a closed Reeb orbit and the zero section. It
``commutes'' with the Hamiltonian action filtration and plays
essentially the same role as the grading by the free homotopy class in
the contact homology of the corresponding $S^1$-bundle. Although the
linking number filtration can be defined in a more general setting, it
is of particular interest to us when the base $B$ is symplectically
aspherical. This is the assumption we will make henceforth. We will
then use the linking number filtration to reprove the non-degenerate
case of the contact Conley conjecture, originally established in
\cite{GGM:CC,GGM:CC2}, without relying on the machinery of contact
homology.

\subsection{Definition of the linking number filtration}
\label{sec:linking-def}
Throughout this section we keep the notation and convention from
Section \ref{sec:prequant}. In particular, $E$ and $M$ are the
prequantization line and, respectively, $S^1$-bundles over a
symplectically aspherical manifold $(B^{2m},\sigma)$, and $f> 1/2$ is
a function on $M$. The domain $W=W_f$ is bounded by the fiberwise star
shaped hypersurface $(1+r^2)/2=f$. This hypersurface $M_f$ has contact
type and the restriction of the primitive $(1+r^2)\alpha_0/2$ (on
$E\setminus B$) to $\p W$ is $\alpha=f\alpha_0$. Furthermore, recall
that all loops and periodic orbits we consider are assumed to be
contractible in $E$ unless stated otherwise.

Assume first that $\alpha$ is non-degenerate and let $H$ be a
time-dependent admissible Hamiltonian on $E=\widehat{W}$. We require $H$ to be
constant on a neighborhood $U$ of $B$ and such that all one-periodic
orbits $x$ of $H$ outside $U$ are small perturbations of closed Reeb
orbits. We will call these orbits non-constant. For a generic choice
of such a Hamiltonian $H$ all non-constant orbits are non-degenerate.

Fix an almost complex structure $J$ on $E$ compatible with $\omega$,
which we require to be independent of time near $B$ and outside a
large compact set, and such that $B$ is an almost complex submanifold
of $E$. Consider solutions $u\colon \R\times S^1\to E$ of the Floer
equation for $(H,J)$ asymptotic as $s\to\pm \infty$ to non-constant
orbits. By the results from \cite{FHS}, the regularity conditions are
satisfied for a generic pair $(H,J)$ meeting the above requirements
and, moreover, for a generic Hamiltonian $H$ as above when $J$ is
fixed. With this in mind, we have the complex $\CF^+(H)$ generated by
non-constant one-periodic orbits of $H$, contractible in $E$, and
equipped with the standard Floer differential. Clearly, the homology
of this complex is $\HF^+(H)$.

The complex $\CF^+(H)$ carries a natural filtration by the linking number
with $B$. Indeed, since $H$ is constant near $B$, every solution $u$ of the
Floer equation is a holomorphic curve near $B$, and the intersection
index of $u$ with $B$ is non-negative since $B$ is an almost complex
submanifold of $E$. When $u$ is a solution connecting $x$ to $y$, the
difference $L_B(x)-L_B(y)$ is exactly this intersection number. Thus
\begin{equation}
\label{eq:decr}
L_B(x)\geq L_B(y)
\end{equation}
and the Floer differential does not increase $L_B$. In other words,
for every $k\in\Z$, the subspace $\CF^+\big(H,\ff^{\leq k}\big)$
generated by the orbits $x$ with $L_B(x)\leq k$ is a subcomplex and we
obtain an increasing filtration of the complex $\CF^+(H)$. Set
$$
\CF^+\big(H,\ff^k\big):= \CF^+\big(H,\ff^{\leq k}\big)/\CF^+\big(H,\ff^{\leq k-1}\big).
$$
We denote the homology of the resulting complexes by
$\HF^+\big(H,\ff^{\leq k}\big)$ and, respectively,
$\HF^+\big(H,\ff^k\big)$.

Passing to the direct limit over $H$, we obtain the homology groups
$\SH^+\big(W,\ff^{\leq k}\big)$ and, $\SH^+\big(W,\ff^{k}\big)$, which
fit into a long exact sequence
$$
\ldots\to \SH^+\big(W,\ff^{\leq k-1}\big)\to \SH^+\big(W,\ff^{\leq
  k}\big)\to \SH^+\big(W,\ff^k\big) \to \ldots.
$$
Furthermore, the complexes $\CF^+\big(H,\ff^{\leq k}\big)$ and
$\CF^+\big(H,\ff^k\big)$ inherit the filtration by the Hamiltonian
action from the complex $\CF(H)$ and this filtration descendes to the
homology groups.

The construction extends to the equivariant setting in a
straightforward way and we only briefly outline it. Following
\cite{BO:Gysin,BO17}, consider a parametrized Hamiltonian
$\tH\colon E\times S^1\times S^{2m+1}\to \R$ invariant with respect to
the diagonal $G=S^1$-action on $S^1\times S^{2m+1}$ and meeting a
certain admissibility condition at infinity. For instance, it is
sufficient to assume that at infinity $\tH$ is independent of
$(t,\zeta)$ and admissible in the standard sense.  We will sometimes
write $\tH_\zeta(t,z)$ for $\tH(t,z,\zeta)$.  Let $\Lambda$ be the
space of contractible loops $S^1\to E$. The Hamiltonian $\tH$ gives
rise to the action functional
$$
\CA_{\tH} \colon \Lambda\times S^{2m+1}\to\R
$$
which is a parametrized version of the standard action
functional. The critical points of $\CA_{\tH}$ are the pairs $(x,\zeta)$
satisfying the condition:
$$
\textrm{the loop $x$ is a one-periodic orbit of $H_\zeta$ and }
\int_{S^1} \nabla_\zeta\tH_\zeta(t,x(t))\,dt=0.
$$
It is convenient to assume that $\tH$ is a small perturbation of an
ordinary non-degenerate Hamiltonian $H$. Then $x$ in such a pair
$(x,\zeta)$ is small perturbation of a one-periodic orbit of $H$.

Due to $S^1$-invariance of $\tH$, the pairs $(x,\zeta)$ come in
families $S$, called \emph{critical families}, even when $\tH$ is
non-degenerate. The complex $\CF^{G}(\tH)$ is generated by the
families $S$. As in the non-equivariant case, one has a subcomplex
$\CF^{-,G}(\tH)$ generated by the critical families $S$ where $x$ is a
constant orbit and the quotient complex $\CF^{+,G}(\tH)$.  The
essential point is that again the Hamiltonian $\tH$ can be taken
constant on $U$ on every slice $E\times(t,\zeta)$. Then the quotient
complex $\CF^{+,G}(\tH)$ is generated by the critical families $S$
such that $x$ is a non-constant one-periodic orbit of $H$.

The parametrized Floer equation has the form
\begin{gather*}
\p_s u+J\p_t u = \nabla_E \tH ,\\
\frac{d\lambda}{ds} = \int_{S^1} \nabla_\zeta \tH (u(t,s),t,\lambda(s))\,dt,
\end{gather*}
where $\lambda\colon \R\to S^{2m+1}$ and $u\colon S^1\times\R\to
E$. Thus, when $\tH$ is constant (e.g., near $B$), $u$ is a
holomorphic curve. It follows that \eqref{eq:decr} holds when
$(u,\lambda)$ connects a critical family containing $x$ to a critical
family containing $y$. As a consequence, the complex $\CF^{+,G}(\tH)$
is again filtered by the linking number. Passing to the limit as
$m\to\infty$ and then over $\tH$ we obtain the linking number
filtration on $\SH^{+,G}(W)$.

We denote the resulting homology groups by
$\SH^{+,G}\big(W,\ff^{\leq k}\big)$ and
$\SH^{+,G}\big(W,\ff^k\big)$. As in the non-equivariant case, these
groups fit into the long exact sequence
$$
\ldots\to \SH^{+,G}\big(W,\ff^{\leq k-1}\big)\to
\SH^{+,G}\big(W,\ff^{\leq k}\big)
\to \SH^{+,G}\big(W,\ff^k\big)
\to\ldots
$$
and also inherit the action filtration. By Lemma
\ref{lemma:action-diff}, this filtration is essentially given by the
contact action up to a constant depending on $k$. Moreover, as is easy
to see, the shift map $D$ from the Gysin sequence respects the linking
number filtration.

It is clear that the linking number filtration is preserved by the
continuation maps because the continuation Hamiltonians can also be
taken constant on $U$. As a consequence, the resulting groups are
independent of the original contact form $\alpha$ or, equivalently,
the domain $W$.

\begin{Remark}
  The construction of the linking number filtration can be generalized
  in several ways. Here we only mention one of them. Let $W$ be a
  compact symplectic manifold with contact type boundary $M$ and let
  $\Sigma\subset W$ be a closed codimension-two symplectic submanifold
  such that its intersection number with any sphere $A\in \pi_1(W)$ is
  zero. Note that for $\Sigma=B\subset E$ this intersection number is,
  up to a factor, $-\left<[\omega],A\right>$, and hence the condition
  is satisfied automatically when $\sigma$ is aspherical. Then for
  every loop $x$ in $M$ contractible in $W$ the linking number
  $L_\Sigma(x)$ is well defined and we have a liking number filtration
  on the (equivariant) positive symplectic homology. This filtration
  has the same general properties as in the particular case discussed
  in this section.
\end{Remark}

\subsection{Calculation of the homology groups}
\label{sec:linking-calc}
The calculation of the linking number filtration groups
$\SH^{+}\big(W,\ff^k\big)$ and $\SH^{+,G}\big(W,\ff^k\big)$ can be
easily carried out by using the standard Morse--Bott type arguments in
Floer homology.

\begin{Proposition}
\label{prop:calc}
Let, as above, $W$ be the prequantization disk bundle over a
symplectically aspherical manifold $(B^{2m},\sigma)$. Then
\begin{itemize}
\item[\rm{(i)}] $\SH^{+}\big(W,\ff^k\big)=0$ and $\SH^{+,G}\big(W,\ff^k\big)=0$ for $k\leq 0$,

\item[\rm{(ii)}] $\SH^{+}\big(W,\ff^k\big)=\H_*(M)[2k-m]$ and

\item[\rm{(iii)}] $\SH^{+,G}\big(W,\ff^k\big)=\H_*(B)[2k-m]$ for $k\in\N$,

\end{itemize}
where all homology groups are taken with rational coefficients.
\end{Proposition}

In particular,
\begin{equation}
\label{eq:top-degree}
\SH^{+,G}_{2k+m}\big(W,\ff^k\big)=\Q \textrm{ for } k\in\N
\end{equation}
and, as expected, $\SH^{+,G}\big(W,\ff^k\big)$ is isomorphic to the
contact homology groups of $(M,\xi)$ for the free homotopy class
$\ff^k$, cf.\ \cite{BO17}.  

\begin{proof}
  Consider an admissible Hamiltonian of the form $H=h(r^2)$, where $h$
  is monotone increasing, convex function equal to zero on
  $[0,1-\eps]$ and to $ar^2+b$ on $[1+\eps,\infty)$, where $\eps>0$ is
  sufficiently small for a fixed $a$. The non-trivial one-periodic
  orbits of $H$ occur on Morse--Bott non-degenerate levels
  $r=r_1,\ldots,r_l$, where $l=\lfloor{a/\pi}\rfloor$, when the form
  $\alpha_0$ is normalized to have integral $\pi$ over the fiber. The
  linking number of the orbits on the level $r=r_k$ with $B$ is
  exactly $k$. Now the proposition follows by the standard Morse--Bott
  argument in Floer homology (see, e.g., \cite{BO:Duke,Poz} and also
  \cite{GG:convex}) together with an index calculation as in, e.g.,
  \cite{GG04}.
\end{proof}

Comparing Case (iii) of Proposition \ref{prop:calc} and Case (ii) of
Proposition \ref{cor:hom-calc}, we see that
\begin{equation}
\label{eq:sum-filtration}
\SH^{+,G}(W)=\bigoplus_{k\in\N}\SH^{+,G}\big(W,\ff^k\big),
\end{equation}
although the isomorphism is not canonical in contrast with
\eqref{eq:cor-tensor}. Note that there is no similar isomorphism in
the non-equivariant case:
$\SH^{+}(W)\neq \bigoplus_{k\in\N}\SH^{+}\big(W,\ff^k\big)$ and, in
fact, the sum on the right is much bigger than $\SH^{+}(W)$.

\begin{Remark}
  Although this is not immediately obvious, one can expect the natural
  maps $\SH^{+,G}\big(W,\ff^{\leq k}\big)\to \SH^{+,G}(W)$ to be
  monomorphisms, resulting in a filtration of $\SH^{+,G}(W)$ by the
  groups $\SH^{+,G}\big(W,\ff^{\leq k}\big)$. Then the right hand side
  of \eqref{eq:sum-filtration} would be the graded space associated
  with this filtration. On the other hand, the decomposition
  \eqref{eq:cor-tensor} gives rise to a similarly looking filtration
  $\bigoplus_{q\leq k}\H_*(B)\otimes \H_{2q}(\CP^\infty)$. However,
  these two filtrations are different. Indeed, the shift operator $D$
  is strictly decreasing with respect to the filtration coming from
  \eqref{eq:cor-tensor} and thus the induced operator on the graded
  space is zero. On the other hand, under the identification
  $\SH^+(W,\ff^k)\cong \H_*(B)$ from Case (ii) of Proposition
  \ref{prop:calc}, the operator $D$ is given by pairing with
  $[\sigma]\in \H^2(B)$ (see \cite[Prop.\ 2.22]{GG:convex}) and this
  pairing is non-trivial. To put this somewhat informally, the
  decompositions \eqref{eq:cor-tensor} and \eqref{eq:sum-filtration}
  do not match term-wise.
\end{Remark}

\begin{Remark}[Lusternik--Schnirelmann inequalities] We can also use
  the linking number filtration to extend the Lusternik--Schnirelmann
  inequalities established in \cite[Thm.\ 3.4]{GG:convex} for exact
  fillings to prequantization bundles. Namely, assume that all closed
  Reeb orbits on $M_\alpha$ are isolated. Then, for any
  $\beta\in \SH^{+,G}(W)$, we have the \emph{strict} inequality
\begin{equation}
\label{eq:LS}
\s(\beta, W_\alpha)>\s(D(\beta), W_\alpha),
\end{equation}
where the right hand side is by definition $-\infty$ when
$D(\beta)=0$. In particular, when the orbits are isolated, 
$$
0< \s^G_{\zeta,0}(W)< \s^G_{\zeta,1}(W)<
\s^G_{\zeta,2}(W)<\ldots
$$
in \eqref{eq:cor-capacities} for every $\zeta\in \H_*(B)$. For the
sake of brevity we only outline the proof of \eqref{eq:LS}. Consider a
``sufficiently large'' admissible autonomous Hamiltonian $H$ constant
on $U$. Then, by \cite[Thm.\ 2.12]{GG:convex}, the strict
Lusternik--Schnirelmann inequality holds for $H$. As a consequence,
there exist two one-periodic orbits $x$ and $y$ of $H$ contractible in
$E$, the carriers for the corresponding action selectors for $\beta$
and $D(\beta)$ in $\HF^{+,G}(H)$, such that $\CA_H(x)>\CA_H(y)$ and
$x$ and $y$ are connected by a solution $u$ of the Floer equation.  As
in the proof of \cite[Thm.\ 3.4]{GG:convex}, we need to show that this
inequality remains strict as we pass to the limit. When $x$ and $y$
are in the same free homotopy class (i.e., $L_B(x)=L_B(y)$), that
proof goes through word-for-word. When, $L_B(x)>L_B(y)$, the Floer
trajectory $u$ has to cross $U$, where it is a holomorphic curve,
passing through a point of $B$. By the standard monotonicity argument,
$\CA_H(x)-\CA_H(y)=E(u)>\eps>0$, where $E(u)$ is the energy of $u$ and
$\eps$ is independent of $H$.

\end{Remark}

\section{Contact Conley conjecture}
\label{sec:CCC}

\subsection{Local symplectic homology}
\label{sec:local}
In this section we recall the definitions of the (equivariant) local
symplectic homology and of symplectically degenerate maxima (SDM) for
Reeb flows -- the ingredients essential for the statement and the proof of
the non-degenerate case of the contact Conley conjecture.

Let $x$ be an isolated closed Reeb orbit of period $T$, not
necessarily simple, for a contact form $\alpha$ on $M^{2m+1}$. The
Reeb vector field coincides with the Hamiltonian vector field of the
Hamiltonian $r$ on $M\times (1-\eps,1+\eps)$ equipped with the symplectic
form $d(r\alpha)$. Consider now the Hamiltonian $H=T\cdot h(r)$, where
$h'(1)=1$ and $h''(1)>0$ is small. On the level $r=1$, this flow is
simply a reparametrization of the Reeb flow and the orbit $x$ 
corresponds to an isolated one-periodic orbit $\tx$ of $H$. By
definition, the \emph{equivariant local symplectic homology}
$\SH^G(x)$ of $x$ is the local $G=S^1$-equivariant Floer homology
$\HF^G(\tx)$ of $\tx$; see \cite[Sect.\ 2.3]{GG:convex}. It is easy to
see that $\SH^G(x):=\HF^G(\tx)$ is independent of the choice of the
function $h$. Note also that this construction is purely local: it
only depends on the germ of $\alpha$ along $x$. In what follows, we will
use the notation $x$ for both orbits $\tx$ and $x$.

These local homology groups do not carry an absolute grading. To fix
such a grading by the Conley--Zehnder index, it is enough to pick a
symplectic trivialization of $T\big(M\times (1-\eps,1+\eps)\big)|_x$.
Depending on a specific setting, there can be different natural ways
to do this. For instance, one can start with a trivialization of the
contact structure $\xi=\ker\alpha$ along $x$; for this trivialization
naturally extends to a trivialization of
$T\big(M\times (1-\eps,1+\eps)\big)|_x$.  However, we are interested
in the situation where $M$ has an aspherical filling $W$ and $x$ is
contractible in $W$. Then it is more convenient to obtain a
trivialization of
$$
T\big(M\times (1-\eps,1+\eps)\big)|_x=TW|_x
$$
from a capping of $x$ in $W$. In any event, when $x$ is iterated, i.e.,
$x=y^k$ where $y$ is simple and also contractible in $W$, we will
always assume that the trivialization of $TW|_x$ comes from a
trivialization along~$y$. This is essential to guarantee that the mean
index is homogeneous under iterations: $\hmu(x)=k\hmu(y)$.

\begin{Example}
\label{ex:non-deg}
  Assume that $x$ is non-degenerate. Then $\SH^G(x)=\Q$, concentrated
  in degree $\mu(x)$, when $x$ is good; and $\SH^G(x)=0$ when $x$ is
  bad; see \cite[Sect.\ 2.3]{GG:convex} and, in particular, Examples
  2.18 and 2.19 therein.
\end{Example}

Furthermore, it is worth keeping in mind that $\tx$ is necessarily
degenerate even when $x$ is non-degenerate. Indeed, the linearized
flow along $\tx$ has 1 as an eigenvalue and its algebraic multiplicity
is at least 2.

As a consequence, $\SH^G(x)$ is supported in the interval of length
$2m$ centered at the mean index $\hmu(x)$ of $x$, i.e., only for the
degrees in this range the homology can be non-zero. (If $\tx$ were
non-degenerate the length of the interval would be $2m+2$.) In other
words, using self-explanatory notation, we have
\begin{equation}
\label{eq:support0}
\supp\SH^G(x)\subset [\hmu(x)-m, \,\hmu(x)+m];
\end{equation}
see \cite[Prop.\ 2.20]{GG:convex}. Moreover,
\begin{equation}
\label{eq:support}
\supp\SH^G(x)\subset (\hmu(x)-m, \,\hmu(x)+m)
\end{equation}
when $x$ is \emph{weakly non-degenerate}, i.e., at least one of its Floquet
multipliers is different from 1.
 
\begin{Remark}
\label{rmk:sympl-contact}
Conjecturally, when $x$ is the $k$th iteration of a simple orbit,
\begin{equation}
\label{eq:sympl-contact}
\SH^G(x)\cong\HC(x)\cong \HF(\varphi)^{\Z_k},
\end{equation}
where $\HC(x)$ is the local contact homology of $x$ introduced in
\cite{HM} (see also \cite{GHHM}), $\varphi$ is the return map of $x$,
and $\HF(\varphi)^{\Z_k}$ is the $\Z_k$-invariant part in the local
Floer homology of $\varphi$ with respect to the natural
$\Z_k$-action. When $x$ is simple, i.e., $k=1$, this has been
proved. Indeed, in this case, $\HC(x)$ is rigorously defined and the
first isomorphism is a local version of the main result in
\cite{BO17}. The second isomorphism is established in \cite{HM} and
can also be thought of as a local variant of the isomorphism between
the Floer and contact homology from \cite{EKP}. When $k\geq 1$, there are
foundational problems with the construction of $\HC(x)$ common to many
versions of the contact homology (see, however, \cite{Ne}) and proving
directly that the first and the last term in \eqref{eq:sympl-contact}
are isomorphic might be a simpler approach. We will return to this
question elsewhere.
\end{Remark}

When $M$ is compact, the groups $\SH^G(x)$, where $x$ ranges over all
closed Reeb orbits of $\alpha$ (not necessarily simple), are the
building blocks for $\SH^{+,G}(W)$ where $W$ is a symplectically
aspherical filling of $M$. For instance, vanishing of the local
homology groups for all $x$ in a certain degree $d(x)$ implies
vanishing of the global (i.e., total) homology in a fixed degree
$d$. However, there might be a shift of degrees, i.e., $d\neq d(x)$,
which depends on the choice of trivializations along the orbits
$x$. This shift is obviously zero when $x$ is contractible in $W$ and
the trivialization of $T_xW$ comes from a capping of $x$. The same
holds for the filtered (by the action or the linking number) homology
groups. In particular, we have

\begin{Lemma} 
\label{lemma:local-to-global}
Let $\alpha$ be a contact form on the prequantization $S^1$-bundle over
a symplectically aspherical manifold $(B^{2m},\sigma)$.  Assume that
all closed Reeb orbits $x$ in the class $\ff^k$ are isolated and
$\SH^G_q(x)=0$ for all such $x$ with respect to the trivialization of
$T_xW$ coming from a capping of $x$ in $W$. Then
$\SH^{+,G}_q\big(W,\ff^k\big)=0$.
\end{Lemma}

The lemma readily follows from the observation that under the above
conditions $\SH^{+,G}_q\big(H,\ff^k\big)=0$ for a suitable cofinal
family of admissible Hamiltonians $H$. (More generally, there is a
spectral sequence starting with $\bigoplus_x\SH^G(x)$ and converging
to $\SH^{+,G}(W)$, which also implies the lemma. We do not need this
fact and we omit its proof for the sake of brevity, for it is quite
standard; see, e.g., \cite{GGM:CC} where such a spectral sequence is
constructed for the contact homology and \cite[Sect.\ 2.2.2]{GG:PR}
for a discussion of this spectral sequence for the Floer homology.)

Next recall that an iteration $k$ of $x$ is called \emph{admissible}
when none of the Floquet multipliers of $x$, different from 1, is a
root of unity of degree $k$. For instance, every $k$ is admissible when
no Floquet multiplier is a root of unity or, as the opposite extreme,
when $x$ is totally degenerate, i.e., all Floquet multipliers are
equal to 1. Furthermore, every sufficiently large prime $k$ (depending
on $x$) is admissible.

For our purposes it is convenient to adopt the following
definition. Namely, $x$ is a \emph{symplectically degenerate maximum (SDM)}
if there exists a sequence of admissible iterations $k_i\to\infty$ such that 
\begin{equation}
\label{eq:sdm}
\SH_q^G(x^{k_i})\neq 0\textrm{ for } q=\hmu(x^{k_i})+m=k_i\hmu(x)+m.
\end{equation}
This condition is obviously independent of the choice of a
trivialization along $x$. It follows from \eqref{eq:support} that then
$x^{k_i}$, and hence $x$, must be totally degenerate. Thus the
definition can be rephrased as that $x$ is totally degenerate and
\eqref{eq:sdm} holds for some sequence $k_i\to\infty$.

\begin{Remark} 
  Continuing the discussion in Remark \ref{rmk:sympl-contact} note
  there are several, hypothetically equivalent, ways to define a
  closed SDM Reeb orbit. The definition above is a contact analog of
  the original definition of a Hamiltonian SDM from \cite{Gi:CC} and
  it lends itself conveniently to the proof of the non-degenerate case
  of the contact Conley conjecture. Alternatively, a contact SDM was
  defined in \cite{GHHM} as a closed isolated Reeb orbit $x$ with
  $\HC_{\hmu(x)+m}(x)\neq 0$. By \eqref{eq:sympl-contact}, its
  symplectic homology analogue would be that
  $\SH^G_{\hmu(x)+m}(x)\neq 0$. For a simple orbit this is equivalent
  to that the fixed point of $\varphi$ is an SDM. Furthermore, one can
  show that the $\Z_k$-action on the homology is trivial for totally
  degenerate orbits and thus
  $\HF(\varphi)^{\Z_k}=\HF(\varphi)$. Hence, the equivalence of the
  two definitions would then follow from the identification of the
  first and the last term in \eqref{eq:sympl-contact} combined with
  the persistence of the local Floer homology; \cite{GG:gap}.
\end{Remark}

\subsection{Conley conjecture}
\label{sec:CC}
Now we are in the position to give a symplectic homology proof of
the non-degenerate case of the contact Conley conjecture, a contact
analogue of the main result from \cite{SZ}.

\begin{Theorem}[Contact Conley Conjecture] 
\label{thm:ccc}
Let $M\to B$ be a prequantization bundle and let $\alpha$ be a contact
form on $M$ supporting the standard (co-oriented) contact structure
$\xi$ on $M$.  Assume that
\begin{itemize}
\item[\rm{(i)}] $B$ is symplectically aspherical and
\item[\rm{(ii)}] $\pi_1(B)$ is torsion free.
\end{itemize}
Then the Reeb flow of $\alpha$ has infinitely many simple closed Reeb
orbits with contractible projections to $B$, provided that none of the
orbits in the free homotopy class $\ff$ of the fiber is an SDM. Assume
in addition that the Reeb flow of $\alpha$ has finitely many closed
Reeb orbits in the class $\ff$. Then for every sufficiently large
prime $k$ the Reeb flow of $\alpha$ has a simple closed orbit in the
class $\ff^k$.
\end{Theorem}

Before proving this theorem let us compare it with other results on
the contact Conley conjecture. Theorem \ref{thm:ccc} was proved in
\cite{GGM:CC,GGM:CC2} without the assumption that none of the orbits
is an SDM. However, that argument relied on the machinery of
linearized contact homology which is yet to be made completely
rigorous. (See, however, \cite{Ne} where some of the foundational
issues have been resolved in dimension three.) The key difficulty in
translating the proof from those two papers into the symplectic
homology framework in the non-degenerate case was purely conceptual:
the grading by the free homotopy classes $\ff^k$ is crucial for the
proof and while the cylindrical contact homology is graded by $\ff^k$,
the symplectic homology is not. The linking number filtration is an
analogue of this grading which allows us to overcome this problem. On
the other hand, removing the ``non-SDM'' assumption requires replacing
the contact homology by the symplectic homology in the main result of
\cite{GHHM}. This is a non-trivial, but technical, issue and we will
return to it elsewhere.

There are also some minor discrepancies between the conditions of
Theorem \ref{thm:ccc} and its counterpart in
\cite{GGM:CC,GGM:CC2}. Namely, there the base $B$ is assumed to be
aspherical, i.e., $\pi_r(B)=0$ for $r\geq 2$, but as is pointed out in
\cite[Sect.\ 2.2]{GGM:CC}, this condition is only used to make sure
that $\sigma$ is aspherical and $\pi_1(B)$ is torsion free. Then, the
class $c_1(\xi)$ is required \emph{ibid} to  be atoroidal. This is a
minor technical restriction imposed only for the sake of simplicity
and it does not arise in the symplectic homology setting because the
fiber is contractible in $W$.

\begin{proof}[Proof of Theorem \ref{thm:ccc}] The argument closely
  follows the reasoning from \cite{GGM:CC}, which in turn is based on
  the proof in \cite{SZ}. We need the following simple, purely
  algebraic fact, proved in \cite[Lemma 4.2]{GGM:CC}, which only uses
  the conditions that $\sigma$ is aspherical and that $\pi_1(B)$ is
  torsion free.

\begin{Lemma}
\label{lemma:torsion-free}
Under the conditions of the theorem, for every $k\in\N$ the only
solutions $\fh\in\tpi_1(P)$ and $l\geq 0$ of the equation
$\fh^l=\ff^k$ are $\fh=\ff^r$, for some $r\in\N$, and $l=k/r$. (In
particular, $\ff$ is primitive.)
\end{Lemma}

Next, without loss of generality we may assume that there are only
finitely many closed Reeb orbits in the class $\ff$, for otherwise
there is nothing to prove. We denote these orbits by $x_1,\ldots,x_r$
and set $\Delta_j=\hmu(x_j)$, where we equipped $T_{x_i}W$ with a
trivialization coming from a capping of $x_i$ in $W$. Let $k$ be a large
prime. Then, unless there is a simple closed Reeb orbit in the class
$\ff^k$, every closed Reeb orbit in this class has the form $x_j^k$ by
Lemma \ref{lemma:torsion-free}.

We will show that in this case $\SH^{+,G}_{m+2k}\big(W,\ff^k\big)=0$
when $k$ is large, which contradicts Proposition \ref{prop:calc} and
more specifically \eqref{eq:top-degree}. By Lemma
\ref{lemma:local-to-global}, it is enough to show that
\begin{equation}
\label{eq:hom=0}
\SH_{m+2k}^G(x^{k}_j)= 0.
\end{equation}
Pick a prime $k$ so large that $k |\Delta_j-2|>2m$ for all $x_j$ with
$\Delta_j\neq 2$. Then, since $\hmu(x_j^k)=k\Delta_j$, we have
$$
\supp\SH^G(x^{k}_j)\subset [k\Delta_j-m,k\Delta_j+m]
$$
by \eqref{eq:support0}, and hence $m+2k$ is not in the support. Thus
\eqref{eq:hom=0} holds in this case. On the other hand, when
$\Delta_j=2$, \eqref{eq:hom=0} holds when $k$ is sufficiently large,
for otherwise \eqref{eq:sdm} would be satisfied for some sequence of
primes $k_i\to\infty$ and $x_j$ would be an SDM. This completes the
proof of the theorem.
\end{proof}


\begin{thebibliography}{CKRTZ}

\bibitem[AS]{AS}
A. Abbondandolo, M. Schwarz, 
Floer homology of cotangent bundles and the loop product,
\emph{Geom.\ Topol.}, \textbf{14} (2010), 
1569--1722. 

\bibitem[AK]{AK} P. Albers, J. Kang, Vanishing of Rabinowitz Floer
  homology on negative line bundles, \emph{Math.\ Z.}, \textbf{285} (2017),
  493--517. 


\bibitem[BO09a]{BO:Duke} 
F. Bourgeois, A. Oancea, 
Symplectic homology, autonomous Hamiltonians, and Morse--Bott moduli
spaces, \emph{Duke Math.\ J.}, \textbf{146} (2009), 
71--174.

\bibitem[BO09b]{BO:Exact} 
F. Bourgeois, A. Oancea,
An exact sequence for contact- and symplectic homology, \emph{Invent.\
Math.}, \textbf{175} (2009), 
611--680.

\bibitem[BO13]{BO:Gysin} 
F. Bourgeois, A. Oancea, 
The Gysin exact sequence for $S^1$-equivariant symplectic homology,
\emph{J. Topol.\ Anal.}, \textbf{5} (2013), 
361--407.

\bibitem[BO17]{BO17} 
F. Bourgeois, A. Oancea,
$S^1$-equivariant symplectic homology and linearized contact homology,
\emph{Int.\ Math.\ Res.\ Not.\ IMRN}, 2017, no.\ 13, 3849--3937.

\bibitem[CF]{CF}
K. Cieliebak, U. Frauenfelder, 
A Floer homology for exact contact embeddings,
\emph{Pacific J. Math.}, \textbf{239} (2009), 251--316.  

\bibitem[CFO]{CFO} K. Cieliebak, U. Frauenfelder, A. Oancea,
Rabinowitz Floer homology and symplectic homology, \emph{Ann.\ Sci.\
\'Ec.\ Norm.\ Sup\'er. (4)}, \textbf{43} (2010),
957--1015. 

\bibitem[CO]{CO} K. Cieliebak, A. Oancea, Symplectic homology and the
  Eilenberg--Steenrod axioms, \emph{Algebr.\ Geom.\ Topol.}, \textbf{18} (2018),
    1953--2130. 

\bibitem[EH]{EH}
I. Ekeland, H. Hofer, 
Symplectic topology and Hamiltonian dynamics, 
\emph{Math.\ Z.}, \textbf{200} (1989), 
355--378. 

\bibitem[EKP]{EKP} Y. Eliashberg, S.S. Kim, L. Polterovich, Geometry
  of contact transformations and domains: orderability vs. squeezing,
  \emph{Geom.\ Topol.}, \textbf{10} (2006), 1635--1747.

\bibitem[EM]{EM} Y. Eliashberg, N. Mishachev, \emph{Introduction to
    the h-Principle}. Graduate Studies in Mathematics, 48. American
  Mathematical Society, Providence, RI, 2002.


\bibitem[FHS]{FHS} A. Floer, H. Hofer, D. Salamon, Transversality
  in elliptic Morse theory for the symplectic action, \emph{Duke
    Math.\ J.}, \textbf{80} (1995), 251--292.

\bibitem[FS]{FS} U. Frauenfelder, F. Schlenk, Hamiltonian dynamics on
  convex symplectic manifolds, \emph{Israel J.\ Math.}, \textbf{159} (2007),
  1--56.

\bibitem[Gi05]{Gi:We} V.L. Ginzburg, The Weinstein conjecture and
  theorems of nearby and almost existence, in \emph{The breadth of
    symplectic and Poisson geometry}, 139--172, Progr.\ Math., 232,
  Birkh\"auser Boston, Boston, MA, 2005.

\bibitem[Gi10]{Gi:CC}
V.L. Ginzburg,
The Conley conjecture, \emph{Ann.\ of Math.} (2), \textbf{172} (2010),
1127--1180.

\bibitem[GG04]{GG04} V.L. Ginzburg, B.Z. G\"urel, Relative
  Hofer-Zehnder capacity and periodic orbits in twisted cotangent
  bundles, \emph{Duke Math.\ J.}, \textbf{123} (2004), 1--47.

\bibitem[GG10]{GG:gap} 
V.L. Ginzburg, B.Z. G\"urel, 
Local Floer homology and the action gap, \emph{J. Sympl.\ Geom.},
\textbf{8} (2010), 323--357.

\bibitem[GG15]{GG:CC} V.L. Ginzburg, B.Z. G\"urel,
The Conley conjecture and beyond, \emph{Arnold Math.\ J.}, \textbf{1}
(2015), 299--337. 

\bibitem[GG16]{GG:convex}
V.L. Ginzburg, B.Z. G\"urel,
Lusternik--Schnirelmann theory and closed Reeb orbits, Preprint arXiv:1601.03092.

\bibitem[GG17]{GG:PR}
V.L. Ginzburg, B.Z. G\"urel,
Hamiltonian pseudo-rotations of projective spaces, Preprint arXiv:1712.09766.


\bibitem[GGM15]{GGM:CC} 
V.L. Ginzburg, B.Z. G\"urel, L. Macarini,
On the Conley conjecture for Reeb flows, 
\emph{Internat.\ J. Math.},
\textbf{26} (2015), 1550047 (22 pages); doi:
10.1142/S0129167X15500470.


\bibitem[GGM17]{GGM:CC2} 
V.L. Ginzburg, B.Z. G\"urel, L. Macarini,
Multiplicity of closed Reeb orbits on prequantization bundles,
Preprint arXiv:1703.04054, to appear in \emph{Israel J. Math.}

\bibitem[GH$^2$M]{GHHM} 
V.L. Ginzburg, D. Hein, U.L. Hryniewicz, L. Macarini, 
Closed Reeb orbits on the sphere and symplectically degenerate maxima,
\emph{Acta Math.\ Vietnam.}, \textbf{38} (2013), 55--78

\bibitem[GGK]{GGK}
V. Guillemin, V. Ginzburg, Y. Karshon, 
\emph{Moment Maps, Cobordisms, and Hamiltonian Group Actions},
Mathematical Surveys and Monographs, \textbf{98}.  American
Mathematical Society, Providence, RI, 2002.

\bibitem[G\"u]{Gu} B.Z. G\"urel, Totally non-coisotropic displacement
  and its applications to Hamiltonian dynamics, \emph{Commun.\
    Contemp.\ Math.}, \textbf{10} (2008), 1103--1128.

\bibitem[Gut]{Gutt}
J. Gutt, The positive equivariant symplectic homology as an invariant
for some contact manifolds, 
\emph{J. Symplectic Geom.}, \textbf{15} (2017), 1019–1069. 

\bibitem[GH]{GH}
J. Gutt, M. Hutchings, Symplectic capacities from
positive $S^1$-equivariant symplectic homology, Preprint arXiv:1707.06514.


\bibitem[HZ]{HZ}
H. Hofer, E. Zehnder,
\emph{Symplectic Invariants and Hamiltonian Dynamics}, Birk\"auser
Verlag, Basel, 1994.

\bibitem[HM]{HM} 
U. Hryniewicz, L. Macarini, 
Local contact homology and applications, \emph{J.\ Topol.\ Anal.},
\textbf{7} (2015), 167--238. 

\bibitem[Hu]{Hu}
  M. Hutchings, Floer homology of families. I, \emph{Algebr.\ Geom.\
    Topol.}, \textbf{8} (2008), 435--492. 

\bibitem[Ir]{Ir}
K. Irie, Private communication, 2017.

\bibitem[LS]{LS}
F. Laudenbach, J.-C. Sikorav, Hamiltonian disjunction and limits of
Lagrangian submanifolds, 
\emph{Int.\ Math.\ Res.\ Not.}, no.\ 4 (1994) 161--168.

\bibitem[Lo]{Lo}
Y. Long,  
\emph{Index Theory for Symplectic Paths with Applications},
Birkh\"auser Verlag, Basel, 2002.

\bibitem[Lu]{Lu}
G. Lu, 
Finiteness of the Hofer-Zehnder capacity of neighborhoods of symplectic submanifolds, 
\emph{Int.\ Math.\ Res.\ Not.}, 2006, Art.\ ID 76520, 1--33.

\bibitem[MS]{MS} L. Macarini, F. Schlenk, A refinement of the
  Hofer-Zehnder theorem on the existence of closed characteristics
  near a hypersurface, \emph{Bull.\ London Math.\ Soc.}, \textbf{37}
  (2005), 297--300. 

\bibitem[McLR]{McLR}
M. McLean, A.F. Ritter, The McKay correspondence via Floer theory,
Preprint arXiv:1802.01534.

\bibitem[Ne]{Ne} J. Nelson, Automatic transversality in contact
  homology II: invariance and computations, Preprint arXiv:1708.07220.

\bibitem[Oa06]{Oa:Kunneth}
A. Oancea, The K\"unneth formula in Floer homology for manifolds with restricted contact type boundary,
\emph{Math.\ Ann.}, \textbf{334} (2006), 65–89.  

\bibitem[Oa08]{Oa:Leray-Serre}
A. Oancea, 
Fibered symplectic cohomology and the Leray--Serre spectral sequence,
\emph{J. Symplectic Geom.}, \textbf{6} (2008), 267--351.  

\bibitem[Po95]{Po:displ}
L. Polterovich, An obstacle to non-Lagrangian intersections, in
\emph{The Floer Memorial Volume}, 575–586,
Progr.\ Math., Vol.\ 133 (Birkh\"auser,
  1995). 

\bibitem[Po98]{Po:Geom}
L. Polterovich, Geometry on the group of Hamiltonian diffeomorphisms, in
\emph{Proceedings of the International Congress of Mathematicians}, Vol.\ II (Berlin, 1998). Doc.\ Math.\ 1998, Extra Vol.\ II, 401--410.

\bibitem[PS]{PS}
L. Polterovich, E. Shelukhin, Autonomous
Hamiltonian flows, Hofer's geometry and persistence modules, \emph{Selecta
Math.\ (N.S.)}, \textbf{22} (2016), 227–296.  


\bibitem[Poz]{Poz} 
M. Po\'zniak,
Floer homology, Novikov rings and clean intersections, in
\emph{Northern California Symplectic Geometry Seminar}, 119--181,
Amer.\ Math.\ Soc.\ Transl.\ Ser.\ 2, \textbf{196}, AMS, Providence,
RI, 1999.

\bibitem[Ri13]{Ri:JT}
A. Ritter, 
Topological quantum field theory structure on symplectic cohomology,
\emph{J. Topol.}, \textbf{6} (2013), 391--489.  

\bibitem[Ri14]{Ri:AM}
A. Ritter, 
Floer theory for negative line bundles via Gromov-Witten invariants, 
\emph{Adv.\ Math.}, \textbf{262} (2014), 1035--1106. 

\bibitem[RS]{RS}
A. Ritter, I. Smith, The monotone wrapped Fukaya category and the
open-closed string map, 
\emph{Sel.\ Math.\ New Ser.}, \textbf{23} (2017), 533–642.

\bibitem[SZ]{SZ}
D. Salamon, E. Zehnder,
Morse theory for periodic solutions of Hamiltonian systems and the
Maslov index, \emph{Comm.\ Pure Appl.\ Math.}, \textbf{45} (1992),
1303--1360.

\bibitem[Schl]{Sc}
F. Schlenk, 
Applications of Hofer's geometry to Hamiltonian dynamics,
\emph{Comment.\ Math.\ Helv.}, \textbf{81} (2006), 105--121.  

\bibitem[Sc]{Sc} M. Schwarz, On the action spectrum for closed
  symplectically aspherical manifolds, \emph{Pacific J. Math.}, \textbf{193}
  (2000), 419--461.

\bibitem[Se]{Se} P. Seidel, A biased view of symplectic
  cohomology, \emph{Current developments in mathematics}, 2006,
  211--253, Int. Press, Somerville, MA, 2008.

\bibitem[Sh]{Sh}
J. Shon, \emph{Filtered Floer Homology, Prequantization Bundles and
  the Contact Conley Conjecture}, Ph.D. Thesis, UCSC, 2018.

\bibitem[Sm]{Sm} I. Smith, Floer cohomology and pencils of quadrics,
\emph{Invent.\ Math.}, \textbf{189} (2012), 149--250. 


\bibitem[Su]{Su}
Y. Sugimoto, Hofer's metric on the space of Lagrangian submanifolds and wrapped Floer homology, 
\emph{J. Fixed Point Theory Appl.}, \textbf{18} (2016), 547--567.  

\bibitem[Ve]{Ve}
S. Venkatesh, Rabinowitz Floer homology and mirror symmetry, Preprint arXiv:1705.00032.

\bibitem[Vi92]{Vi:gen} C. Viterbo, Symplectic topology as the geometry of
  generating functions, \emph{Math.\ Ann}. \textbf{292} (1992), 
  685--710. 


\bibitem[Vi99]{Vi:GAFA} 
C. Viterbo, 
Functors and computations in Floer cohomology, I, \emph{Geom.\ Funct.\
Anal.}, \textbf{9} (1999), 985--1033.


\bibitem[UZ]{UZ} M. Usher, J. Zhang, Persistent homology and
  Floer--Novikov theory, \emph{Geom.\ Topol.}, \textbf{20} (2016)
  3333--3430.


\end{thebibliography}
\end{document}